\documentclass[a4paper]{amsart}

\usepackage[a4paper,width={16cm},left=2.5cm,bottom=3cm, top=3cm]{geometry}
\usepackage{enumerate}

\usepackage[utf8]{inputenc}
\usepackage[english]{babel}

\usepackage{a4wide}
\usepackage{xcolor}
\usepackage{amsmath}
\usepackage{amssymb}
\usepackage{amsfonts}
\usepackage{amsthm,graphicx}

\usepackage{hyperref}
\usepackage{mathtools}
\mathtoolsset{showonlyrefs}

\usepackage{scalerel,stackengine}
\stackMath
\newcommand\reallywidehat[1]{
\savestack{\tmpbox}{\stretchto{\scaleto{\scalerel*[\widthof{\ensuremath{#1}}]{\kern-.6pt\bigwedge\kern-.6pt}{\rule[-\textheight/2]{1ex}{\textheight}}}{\textheight}}{0.5ex}}\stackon[1pt]{#1}{\tmpbox}
}

\newtheorem{theorem}{Theorem}[section]
\newtheorem{lemma}[theorem]{Lemma}
\newtheorem{corollary}[theorem]{Corollary}

\newtheorem{proposition}[theorem]{Proposition}
\theoremstyle{definition}

\newtheorem{remark} [theorem] {Remark}

\newcommand{\la}{\lambda}
\newcommand{\norm}[1]{\left\lVert#1\right\rVert}
\newcommand{\pd}[2]{\frac{\partial#1}{\partial#2}}
\newcommand\restr[2]{\ensuremath{\left.#1\right|_{#2}}}
\newcommand{\R}{\mathbb{R}}

\newcommand{\e}{\varepsilon}

\newcommand{\Tr}{\mathop{\rm{Tr}}}
\newcommand{\dive}{\mathop{\rm{div}}}
\newcommand{\df}[4]{\ensuremath\sideset{_{#1}}{_{#4}}{\mathop{\left\langle #2, #3 \right\rangle}}}

\begin{document}

\title[Strong unique continuation from the boundary]
{Strong unique continuation from the boundary\\ for the spectral fractional Laplacian}

\author{Alessandra De Luca, Veronica Felli, and Giovanni Siclari}
\address{Alessandra De Luca 
\newline \indent Dipartimento di Scienze Molecolari e Nanosistemi,
Università Ca' Foscari Venezia,
\newline \indent  Via Torino 155, 30172 Venezia Mestre, Italy.}
\email{alessandra.deluca@unive.it}

\address{Veronica Felli and Giovanni Siclari 
\newline \indent Dipartimento di Matematica e Applicazioni, Università
degli Studi di Milano-Bicocca,
\newline \indent Via Cozzi 55, 20125 Milano, Italy.}
\email{veronica.felli@unimib.it,  g.siclari2@campus.unimib.it}

\date{January 27, 2023}
\thanks{The authors are partially supported by the INDAM-GNAMPA
    2022 grant ``Questioni di esistenza e unicit\`a per problemi non
    locali con potenziali''.
Part of this work was carried out while A. De Luca and V. Felli were
participating in the research program ``Geometric Aspects of Nonlinear
Partial Differential Equations'' at Institut Mittag-Leffler in Djursholm, Sweden, in 2022.}

\begin{abstract}
  We investigate unique continuation properties and asymptotic
    behaviour at boundary points for solutions to a class of elliptic
    equations involving the spectral fractional Laplacian.  An
    extension procedure leads us to study a degenerate or singular
    equation on a cylinder, with a homogeneous Dirichlet boundary
    condition on the lateral surface and a non homogeneous Neumann
    condition on the basis. For the extended problem, by an
    Almgren-type monotonicity formula and a blow-up analysis, we
    classify the local asymptotic profiles at the edge where the
    transition between boundary conditions occurs. Passing to
    traces, an analogous blow-up result and its consequent strong
    unique continuation property is deduced for the nonlocal
    fractional equation.
\end{abstract}

\maketitle



\section{Introduction and statement of the main results}

In this paper we prove the strong unique continuation property and
derive local asymptotics from the boundary for the solutions to the
following equation
\begin{equation}\label{eq-spec-lapla}
	(-\Delta)^{s}u= h u  \quad \text{ on } \Omega,
\end{equation}
where $s \in (0,1)$, $\Omega \subseteq \R^N$ is a bounded Lipschitz domain with
$N >2s$, $h$ is a measurable function on $\Omega$ satisfying suitable
summability properties which will be more specifically clarified below
(see \eqref{hypo-h}) and $(-\Delta)^s$ is the so-called
\emph{spectral} fractional Laplacian.
 
The unique continuation property has been abundantly studied over the
years for several problems.  We recall that a
family of functions, including the zero function, satisfies the
\emph{strong unique continuation property} if the null function is
the only one to have a zero of infinite order.

Several results are available in the literature about the spectral
fractional Laplacian and its interpretations. See \cite{AD},
\cite{LAPGGM}, and references therein for a detailed overview.
We mention that regularity properties for stationary equations
are discussed in \cite{G}, while existence and uniqueness results for
evolution equations governed by the spectral fractional Laplacian are
established in \cite{BY}.  More closely related to the present paper
are the results in \cite{YU}, where a strong unique continuation
principle at nodal points is proved for fractional powers
  of some divergence-type elliptic operators, including the case of
the spectral
fractional Laplacian. The techniques used in \cite{YU} are
inspired by those
introduced in \cite{Fall-Felli-2014}, which are based
on a combination of a monotonicity formula for an Almgren-type
frequency function and a blow up analysis. This local approach is made
possible by the extension result \cite[Theorem 2.5]{CS} due to
Caffarelli and Stinga, see also \cite{stinga-torrea}.

The development of a monotonicity formula for the extended
  problem presents new difficulties when dealing with boundary
  points. Indeed, since the point $x_0$ from which the unique
  continuation is sought after lies on $\partial \Omega$, the
geometry of $\partial \Omega$ can interfere with the monotonicity
argument.  In the present paper, we face this difficulty by
  straightening the boundary with a local diffeomorphism that
  transfers the information about the geometry of $\partial\Omega$
  into a coefficient matrix in the operator, which turns out to be a
  perturbation of the identity if the boundary is regular enough, see
Section \ref{sec-straightening the boundary}. Secondly, a Pohozaev
type identity is needed to differentiate the frequency function
  and to develop the monotonicity argument. To this aim, we
  rely on a more general result contained in \cite[Proposition
2.3]{fellisiclari}, which is based on a Sobolev-type regularity
  theory for a class of degenerate and singular problems.
Furthermore, a blow-up analysis provides a detailed description of
  the asymptotic behaviour of solutions to \eqref{eq-spec-lapla} at
  $x_0$, giving a complete classification of the order of
  homogeneity of asymptotic profiles, see Theorem
  \ref{theor-blow-up-down} below.  For this purpose, an important role
    is played by an eigenvalue problem on a half-sphere
  under a symmetry condition, see \eqref{prob-eigenvalues}.

  The extension problem corresponding to \eqref{eq-spec-lapla}
    consists of a degenerate or singular equation on the cylinder
    $\Omega\times(0,+\infty)$, with a homogeneous Dirichlet boundary
    condition on the lateral surface $\partial\Omega\times(0,+\infty)$
    and a Neumann derivative on the basis $\Omega\times\{0\}$ being
    equal to the right hand side of \eqref{eq-spec-lapla}, see
    \eqref{prob-extension}. Therefore, the formulation of the problem
    in terms of the extension leads us to study what happens near a
    point of the edge at which a transition between boundary
    conditions of a different type takes place. We observe that this
    situation is quite different from the one that occurs in
    \cite{de2021strong}, where unique continuation from boundary
    points is studied for the \emph{restricted fractional Laplacian};
    indeed the extension problem corresponding to the case treated in
    \cite{de2021strong} is a degenerate or singular problem with mixed
    conditions that vary on a flat basis rather than on an edge. In
    fact, the analysis carried out in the present paper highlights
    different asymptotic behaviors at the boundary for the two
    operators, unlike what happens at internal points, where the
    locally equivalent form of the extended problems induces the same
    blow-up profiles.

  In order to introduce a suitable functional setting and give a weak
  formulation of \eqref{eq-spec-lapla}, we recall the definition of
  the spectral fractional Laplacian, which can be given in terms
    of the Dirichlet eigenvalues of the Laplacian, see e.g.
  \cite{MR2825595}, \cite{LAPGGM} and \cite{AD}.  From classical
  spectral theory, the Dirichlet eigenvalue problem
\begin{equation*}
\begin{cases} -\Delta \varphi =\mu \varphi ,  &\text{in } \Omega, \\
 \varphi =0, \quad &\text{on } \partial \Omega,
\end{cases}	
\end{equation*}
admits an increasing and diverging sequence
$\{\mu_k\}_{k \in \mathbb{N}\setminus\{0\}}$ of positive
eigenvalues (repeated according to their
  multiplicity). Furthermore, there exists an orthonormal basis of
$L^2(\Omega)$ made of the corresponding eigenfunctions
$\{\varphi_k\}_{k \in \mathbb{N}\setminus\{0\}}$. Every
$v \in L^2(\Omega)$ can be expanded with respect to the basis
$\{\varphi_k\}_{k \in \mathbb{N}\setminus\{0\}}$ as 
\begin{equation*}
  v=\sum_{k=1}^\infty (v,\varphi_k)_{L^2(\Omega)}
  \varphi_k\quad\text{in $L^2(\Omega)$},
\end{equation*}
where $(v,\varphi_k)_{L^2(\Omega)}$ is the $L^2$-scalar product,
i.e. $(v_1,v_2)_{L^2(\Omega)}=\int_{\Omega} v_1 v_2 \, dx$.

We introduce  the functional space 
\begin{equation*}
  \mathbb{H}^s(\Omega):=\left\{v \in L^2(\Omega):
    \sum_{k=1}^\infty\mu^s_k
    ( v, \varphi_k)_{L^2(\Omega)}^2
    <+\infty\right\}
\end{equation*}
which is  a Hilbert space with
respect to the scalar product
\begin{equation}\label{fract-lapla-domain-scalar-product}
 (v_1,v_2)_{\mathbb{H}^s(\Omega)}:=
  \sum_{k=0}^\infty\mu^s_k
( v_1, \varphi_k)_{L^2(\Omega)}
( v_2, \varphi_k)_{L^2(\Omega)}
 ,	\quad v_1,v_2 \in \mathbb{H}^s(\Omega).
\end{equation}
A more explicit characterization of the space
$\mathbb{H}^s(\Omega)$ is provided by  the interpolation theory, see
\cite[Section 3.1.3]{BY} and \cite{MR0350177}:
\begin{equation*}
\mathbb{H}^s(\Omega)=[H^1_0(\Omega),L^2(\Omega)]_{1-s}=
\begin{cases}
H^s_0(\Omega), &\text{if } s \in(0,1)\setminus\{\frac{1}{2}\}, \\
H_{00}^{1/2}(\Omega), &\text{if } s =\frac{1}{2}.
\end{cases}	
\end{equation*}
Here, denoting as $H^s(\Omega)$ the usual fractional Sobolev space
$W^{s,2}(\Omega)$, $H_0^s(\Omega)$ is the closure of
$C^\infty_c(\Omega)$ in $H^s(\Omega)$, and
\begin{equation*}
  H_{00}^{1/2}(\Omega):=\left\{u \in H_0^{\frac12}(\Omega):
    \int_{\Omega}\frac{u^2(x)}{d(x,\partial \Omega)}\, dx <+\infty\right\},	
\end{equation*}
where $d(x,\partial \Omega):=\inf\{|x-y|:y \in \partial \Omega\}$. We
recall that $H^s(\Omega)=H^s_0(\Omega)$ if $s \in (0,\frac{1}{2}]$,
see \cite{MR0350177}. Moreover, if $s\neq\frac12$, the trivial
  extension by $0$ outside $\Omega$ defines a linear and continuous
  operator from $H^s_0(\Omega)$ into $H^s(\R^N)$, see \cite[Remark 2.5
  and Proposition B.1]{BLP}. On the other hand, the trivial extension
  defines a linear and continuous operator from $H_{00}^{1/2}(\Omega)$
  into $H^{1/2}(\R^N)$, as one can easily deduce from estimate (B.2)
  in \cite{BLP}. Then
  \begin{align}\label{eq:triv-ext}
    \iota: \mathbb{H}^s(\Omega)&\to H^s(\R^N),\\
   \notag v&\mapsto \tilde v=
       \begin{cases}
         v,&\text{in }\Omega,\\
         0,&\text{in }\R^N\setminus\Omega,
       \end{cases}
  \end{align}
 is a linear and continuous operator.

It is easy to verify that, if $v\in \mathbb{H}^s(\Omega)$, then
  the series
  $\sum_{k=1}^\infty\mu^s_k (v,\varphi_k)_{L^2(\Omega)} \varphi_k$
  converges in the dual space $(\mathbb{H}^s(\Omega))^*$ to some
  $F\in (\mathbb{H}^s(\Omega))^*$ such that
  $\df{(\mathbb{H}^s(\Omega))^*}{F}{\varphi_k}{\mathbb{H}^s(\Omega)}=\mu^s_k
  (v,\varphi_k)_{L^2(\Omega)}$. Hence, for every
  $v\in \mathbb{H}^s(\Omega)$, we can define its spectral fractional
  Laplacian as
\begin{equation}\label{fract-lapla}
  (-\Delta)^s v = \sum_{k=1}^\infty\mu^s_k
  (v,\varphi_k)_{L^2(\Omega)}
  \varphi_k \in (\mathbb{H}^s(\Omega))^*.
\end{equation}
Actually, the spectral fractional Laplacian is the Riesz isomorphism
between $\mathbb{H}^s(\Omega)$ endowed with the scalar product
  \eqref{fract-lapla-domain-scalar-product} and its dual
$(\mathbb{H}^s(\Omega))^*$, i.e.
\begin{equation}\label{frac-lapla-ritz}
  \df{(\mathbb{H}^s(\Omega))^*}
  {(-\Delta)^s v_1}{v_2}{\mathbb{H}^s(\Omega)}=
  (v_1,v_2)_{\mathbb{H}^s(\Omega)}
  \quad \text{ for all } v_1,v_2 \in\mathbb{H}^s(\Omega).
\end{equation}
The spectral fractional Laplacian defined in \eqref{fract-lapla} is a
different operator from the usual fractional Laplacian defined by the
Fourier transformation as
\begin{equation}\label{eq-frac-laplacain}
\mathcal{F}{( (-\Delta)^s v)}(\xi):=|\xi|^{2s}\widehat{v}(\xi)
\end{equation} 
for any $v \in \mathcal{S}(\R^N)$.  Indeed, the spectral fractional
Laplacian depends on the domain $\Omega$ and it is a global operator
in $\Omega$, while the fractional Laplacian is a global operator on
the whole $\R^N$. Moreover, the eigenfunctions of the spectral
fractional Laplacian coincide with the eigenfunctions of the
  Dirichlet Laplacian, hence they are smooth up to the boundary if
  $\Omega$ is sufficiently regular; on the other hand, the
eigenfunctions of the restricted fractional Laplacian,
defined by restricting the operator in \eqref{eq-frac-laplacain}
  to act only on functions vanishing outside $\Omega$, are only
H\"older continuous, see \cite{SV}.

Within the functional setting introduced above, we can give the
  notion of weak solution to \eqref{eq-spec-lapla}.  To this
purpose, we assume
that
\begin{equation}\label{hypo-h}
h \in W^{1,\frac{N}{2s}+\e}(\Omega)
\end{equation}
for some $\e\in(0,1)$. We note that it is not restrictive to assume
  $\varepsilon$ small.  In view of \eqref{frac-lapla-ritz}, we say
that a function $u \in \mathbb{H}^s(\Omega)$ is a weak solution to
\eqref{eq-spec-lapla} if
\begin{equation}\label{eq-spec-lapla-weak}
  (u,\phi )_{\mathbb{H}^s(\Omega)}=
  \int_{\Omega} h(x)u(x) \phi(x) \, dx
  \quad \text{ for any } \phi \in C^\infty_c(\Omega).
\end{equation}
The right hand side in \eqref{eq-spec-lapla-weak} is finite in
  view of \eqref{hypo-h}, the H\"older's inequality, and the
following fractional Sobolev inequality
\begin{equation*}
  \norm{v}_{L^{2^*_s}(\Omega)}\le \mathcal{K}_{N,s}
  \norm{v}_{H^s(\Omega)}
  \quad \text{ for any } v \in H_0^s(\Omega),	
\end{equation*}
where 
\begin{equation}\label{2ast}
2^*_s:=\frac{2N}{N-2s},
\end{equation}
and $ \mathcal{K}_{N,s}>0$ is a positive constant depending only on
$N$ and $s$, see e.g. \cite[Theorem 6.5]{DNPV} and \cite[Remark 2.5 and
  Proposition B.1]{BLP}.

In order to establish a unique continuation property at a fixed point
$x_0 \in \partial \Omega$, we need to assume some regularity on the
boundary of $\Omega$ near $x_0$; more precisely, we assume that there
exist a radius $R>0$ and a function $g$ such that
\begin{equation}\label{hypo-g}
	g\in C^{1,1}(\R^{N-1},\R)
\end{equation}
and, up to rigid motions, letting $x=(x',x_N) \in \R^{N-1} \times \R$,
\begin{align} 
  &\partial \Omega \cap B'_R(x_0) = \{(x',x_N)
    \in B'_R(x_0):x_N=g(x')\}, \label{hypo-g-boundary-Omega}\\
  &\Omega	 \cap B'_R(x_0) = \{(x',x_N)
    \in B'_R(x_0): x_N<g(x')\}, \label{hypo-g-Omega}
\end{align}
where, for any $r >0$ and $x \in \R^N$,
\begin{equation}\label{Brx}
B'_r(x):=\{y \in \R^N:|y-x|<r\}.
\end{equation}
The spectral fractional Laplacian defined in \eqref{fract-lapla} turns
out to be a nonlocal operator on $\Omega$. As we intend to use an
  approach based on local doubling inequalities, which are deduced
from an Almgren-type monotonicity formula in the spirit of
\cite{MR833393}, it is quite natural to deal with the local
  realization of the spectral fractional Laplacian. This is obtained by
the extension procedure described in
\cite{CS} (see also \cite{stinga-torrea} and \cite{MR2825595}) which
transforms \eqref{eq-spec-lapla} into a singular or degenerate problem
on a cylinder contained in a $N+1$-dimensional space.

More precisely, we consider the half-space 
$\R^{N+1}_+:=\R^N \times(0,\infty)$,
whose total variable is denoted as
$z=(x,t) \in \R^{N} \times [0,\infty)$.  For any open set
$E \subseteq \R^{N} \times (0,\infty)$, let $H^1(E,t^{1-2s})$ be the
completion of $C^\infty_c(\overline E)$ with respect to the norm
\begin{equation*}
  \norm{\phi}_{H^1(E,t^{1-2s})}:=
  \left(\int_{E} t^{1-2s}(\phi^2 +|\nabla \phi|^2) \, dz\right)^\frac12.
\end{equation*}
By \cite[Theorem 11.11, Theorem 11.2, 11.12 Remarks(iii)]{MR802206}
and the extension theorems for weighted Sobolev spaces with weights in
the Muckenhoupt's $A_2$ class proved in \cite{MR1245837}, for any open
Lipschitz set $E\subseteq \R^{N} \times (0,\infty)$,
the space $H^1(E,t^{1-2s})$ can be characterized as
\begin{equation*}
  H^1(E,t^{1-2s})=  \left\{v \in W^{1,1}_{\rm loc}(E):
    \int_{E} t^{1-2s} (v^2+|\nabla v|^2)\, dz< +\infty\right\}.
\end{equation*}
We define
\begin{equation}\label{C+}
  \mathcal{C}_{\Omega}:= \Omega \times (0,+\infty), \quad
\partial_L\mathcal{C}_{\Omega}:= \partial \Omega  \times
[0,+\infty),
\end{equation}
and 
\begin{equation*}
  H_{0,L}^1(\mathcal{C}_{\Omega},t^{1-2s})
  :=\overline{\{\phi \in
    C^{\infty}_c(\overline{\mathcal{C}_{\Omega}}):\phi
    =0 \text{ on } 
    \partial_L\mathcal{C}_{\Omega}\}}^{\norm{\cdot}_{H^1(\mathcal{C}_{\Omega},t^{1-2s})}}, 
\end{equation*}
i.e. $H_{0,L}^1(\mathcal{C}_{\Omega},t^{1-2s})$ is the closure in
$H^1(\mathcal{C}_{\Omega},t^{1-2s})$ of
$\{\phi \in C^{\infty}_c(\overline{\mathcal{C}_{\Omega}}):\phi=0
\text{ on } \partial_L\mathcal{C}_{\Omega}\}$.
Furthermore there exists a linear and continuous trace operator 
\begin{equation}\label{trace-Omega}
  \mathop{\rm{Tr}_\Omega}:
  H_{0,L}^1(\mathcal\mathcal{C}_{\Omega},t^{1-2s})
  \to  \mathbb{H}^s(\Omega)
\end{equation}
which is also onto (see \cite[Proposition 2.1]{MR2825595}).
Moreover,
in \cite{MR2825595} it is observed that, for every
$v \in~\!\mathbb{H}^s(\Omega)$, the minimization problem
\begin{equation*}
  \min_{\substack{w\in H_{0,L}^1(\mathcal{C}_{\Omega},t^{1-2s})\\
      \mathop{\rm{Tr}_\Omega}(w)=v}}\left\{
    \int_{\mathcal{C}_{\Omega}}t^{1-2s} |\nabla w(x, t)|^2   \, dx\,dt\right\}
\end{equation*}
has a unique minimizer
$\mathcal H(v)=V \in H_{0,L}^1(\mathcal{C}_{\Omega},t^{1-2s})$ which solves
\begin{equation}\label{prob-extension-general}
\begin{cases}
\dive(t^{1-2s} \nabla V)=0, &\text{in } \mathcal{C}_{\Omega}, \\
\mathop{\rm{Tr}_\Omega}(V)=v,  & \text{on } \Omega\times \{0\}, \\
V=0,  & \text{on } \partial\Omega\times[0,+\infty), \\
-\lim_{t \to 0^+} t^{1-2s} \pd{V}{t}= \kappa_{s,N}(-\Delta )^sv,
& \text{on } \Omega\times \{0\},
\end{cases}
\end{equation}
where $\kappa_{s,N} >0$ is a positive constant depending only on $N$
and $s$. Equation \eqref{prob-extension-general} has to be interpreted
in a weak sense, that is
\begin{equation}\label{eq-weak-extension}
  \int_{\mathcal{C}_{\Omega}}t^{1-2s} \nabla V \cdot \nabla \phi \, dz
  =\kappa_{s,N}
  (v,\mathop{\rm{Tr}_\Omega}(\phi))_{\mathbb{H}^s(\Omega)}
  \quad \text{for all } \phi
  \in  H_{0,L}^1(\mathcal{C}_{\Omega},t^{1-2s}),
\end{equation}
in view of \eqref{frac-lapla-ritz}.  Hence, \label{ext} if
$u \in \mathbb{H}^s(\Omega)$ solves \eqref{eq-spec-lapla}, then its
extension
$\mathcal H(u)=U \in H_{0,L}^1(\mathcal{C}_\Omega,t^{1-2s})$
weakly solves
\begin{equation}\label{prob-extension}
\begin{cases}
  \dive(t^{1-2s} \nabla U)=0, &\text{in } \mathcal{C}_\Omega, \\
  \mathop{\rm{Tr}_\Omega}(U)=u,  & \text{on } \Omega\times \{0\}, \\
U=0,  & \text{on } \partial\Omega\times[0,+\infty), \\
  -\lim_{t \to 0^+} t^{1-2s} \pd{U}{t}= \kappa_{s,N}hu, & \text{on }
  \Omega\times \{0\},
\end{cases}	
\end{equation}
according to \eqref{prob-extension-general},
namely
\begin{equation}\label{eq-weak}
  \int_{\mathcal{C}_\Omega}t^{1-2s} \nabla U \cdot \nabla \phi \, dz
  =\kappa_{s,N}\int_\Omega hu \mathop{\rm{Tr}_\Omega}(\phi) \, dx
  \quad
  \text{for all }   \phi \in  H_{0,L}^1(\mathcal{C}_\Omega,t^{1-2s}).
\end{equation}
The asymptotic behavior at $x_0 \in \partial \Omega$ of any solution
$U$ of \eqref{prob-extension}, and consequently of any solution $u$ of
\eqref{eq-spec-lapla}, turns out to be related to the eigenvalues of
the following problem
\begin{equation}\label{prob-eigenvalues}
\begin{cases}
  -\mathop{\rm{div}_{\mathbb{S}}}(\theta_{N+1}^{1-2s}\nabla_{\mathbb{S}}Y)=
  \mu \,\theta_{N+1}^{1-2s} \,Y, &\text{on } \mathbb{S}^+\\[5pt]
  \lim_{\theta_{N+1} \to 0^+}\theta_{N+1}^{1-2s}\,\nabla_{\mathbb{S}}
  Y\cdot\nu=0,
  &\text{on } \mathbb{S}',\\[5pt]
  Y\in H_{\rm odd}^1(\mathbb{S}^+,\theta_{N+1}^{1-2s}),
\end{cases}	
\end{equation}
where 
\begin{align*}
&\mathbb{S}:=\{\theta=(\theta', \theta_N,\theta_{N+1}) \in \R^{N+1}:|\theta'|^2+\theta_N^2+\theta_{N+1}^2=1\},\\
&\mathbb{S}^+:=\{\theta=(\theta', \theta_N,\theta_{N+1}) \in \mathbb{S}: \theta_{N+1}>0\},\\
&\mathbb{S}':=\partial \mathbb{S}^+=\{\theta=(\theta',\theta_N, \theta_{N+1}) \in \mathbb{S}: \theta_{N+1}=0\},
\end{align*}
and 
$\nu$ is the outer normal vector to $\mathbb{S}^+$ on $\mathbb{S}'$, that is $\nu=-(0,\dots, 0,1)$.
We consider  the weighted space 
\begin{equation*}
  L^2(\mathbb{S}^+,\theta_{N+1}^{1-2s}):=\left\{\Psi:\mathbb{S}^+ \to
    \R \text{ measurable}: \int_{\mathbb{S}^+}\theta_{N+1}^{1-2s}
    \Psi^2 \, dS <+\infty\right\},
\end{equation*}
where $dS$ denotes the volume element on $N$-dimensional spheres.
In order to introduce the space
$H_{\rm odd}^1(\mathbb{S}^+,\theta_{N+1}^{1-2s})$ in which problem
\eqref{prob-eigenvalues} is formulated, we first denote by
$H^1(\mathbb{S}^+,\theta_{N+1}^{1-2s})$ the completion of
$C^{\infty}(\overline{\mathbb{S}^+})$ with respect to the norm
\begin{equation*}
  \norm{\phi}_{H^1(\mathbb{S}^+,\theta_{N+1}^{1-2s})}:
  =\left(\int_{\mathbb{S}^+}\theta_{N+1}^{1-2s}( \phi^2
    +|\nabla_{\mathbb{S}} \phi|^2 )\, dS\right)^{1/2}.
\end{equation*}
Then we define
\begin{equation}\label{H1d}
  H_{\rm odd}^1(\mathbb{S}^+,\theta_{N+1}^{1-2s}):=\\
  \{\Psi \in
  H^1(\mathbb{S}^+,\theta_{N+1}^{1-2s}):\Psi(\theta',\theta_N,\theta_{N+1})
  =-\Psi(\theta',-\theta_N,\theta_{N+1})\}.
\end{equation}
It is easy to verify that $H_{\rm odd}^1(\mathbb{S}^+,\theta_{N+1}^{1-2s})$
is a closed subspace of $ H^1(\mathbb{S}^+,\theta_{N+1}^{1-2s})$.

A function $Y \in H_{\rm odd}^1(\mathbb{S}^+,\theta_{N+1}^{1-2s})$ is
an eigenfunction of \eqref{prob-eigenvalues} if $Y\not\equiv0$ and
\begin{equation}\label{eq-egienvlulues}
  \int_{\mathbb{S}^+} \theta_{N+1}^{1-2s}\, \nabla_{\mathbb{S}}Y \cdot
  \nabla_{\mathbb{S}} \Psi \, dS
  = \mu \int_{\mathbb{S}^+} \theta_{N+1}^{1-2s}  Y \Psi \, dS
\end{equation}
for all $\Psi \in H_{\rm odd}^1(\mathbb{S}^+,\theta_{N+1}^{1-2s})$.

By classical spectral theory, the set of the eigenvalues of problem
\eqref{prob-eigenvalues} is an increasing and diverging sequence of
positive real numbers $\{\mu_m\}_{m
  \in\mathbb{N}\setminus\{0\}}$. In Appendix
\ref{appendix-eigenvalues-half-sphere} we explicitly determine the
sequence $\{\mu_m\}_{m \in\mathbb{N}\setminus\{0\}}$, obtaining that,
for all $m \in \mathbb{N}\setminus\{0\}$,
\begin{equation}\label{eigenvalues}
\mu_m=
\begin{cases}
m^2+m(N-2s), &\text{if } N>1, \\
(2m-1)^2+(2m-1)(N-2s), &\text{if } N=1.
\end{cases}
\end{equation} 
Let, for future reference,  
\begin{align}
  &V_m \text{ be the eigenspace of problem \eqref{prob-eigenvalues}
    associated to the eigenvalue } \mu_m,\label{Eigenspaces}\\
  &M_m \text{ be the dimension of } V_m,
    \label{dimension-egigenspaces}	\\
  &\{Y_{m,k}:m \in \mathbb{N}\setminus \{0\} \text{ and }
    k \in \{1, \dots,M_m\}\} \text{ be an orthonormal basis of }
    L^2(\mathbb{S}^+,\theta_{N+1}^{1-2s}) \label{orthonormal-base} \\
  &\text{ such that } \{Y_{m,k}: k=1, \dots, M_m\}
    \text{ is a basis  of } V_m. \notag
\end{align}

\begin{remark}\label{rem-Y-not-0}
  Let $Y$ be an eigenfunction of \eqref{prob-eigenvalues} associated
  to the eigenvalue $m^2+m(N-2s)$. Then $Y$ can not vanish identically
  on $\mathbb{S'}$.

  Indeed, if $Y\equiv0$ on $\mathbb{S'}$, we
    would have that $V(r\theta):=r^{m}Y(\theta)$ would solve
  $\dive(t^{1-2s}\nabla V)=0$ on $\R_+^{N+1}$, satisfying both
  Neumann and Dirichlet boundary condition on $\R^N \times \{0\}$.
  This would contradict the unique continuation
  principle for elliptic equations with weights in the  Muckenhoupt
  $A_2$ class, see \cite{MR833393}, \cite{MR2370633}, and
  \cite[Proposition 2.2]{MR3268922}. 
\end{remark}

The main result of the present paper is a complete classification
  of asymptotic blow-up profiles at a point $x_0 \in \partial \Omega$
  for solutions of \eqref{prob-extension-general} and, in turn, for
the corresponding solutions of \eqref{eq-spec-lapla}.

\begin{theorem}\label{theor-blow-up-down}
  Let $N>2s$ and $\Omega\subset \R^N$ be a bounded Lipschitz domain. Let
  $x_0 \in \partial \Omega$ and assume that there exist $R>0$ and a
  function $g$ satisfying \eqref{hypo-g},
  \eqref{hypo-g-boundary-Omega}, and \eqref{hypo-g-Omega}.  Let $u$ be
  a non trivial solution of \eqref{eq-spec-lapla} in the sense of
  \eqref{eq-spec-lapla-weak}, with $h$ satisfying \eqref{hypo-h}. Then
  there exists $m_0 \in \mathbb{N} \setminus \{0\}$ (which is odd
    in the case $N=1$) and an eigenfunction $Y$ of
  \eqref{prob-eigenvalues} associated to the eigenvalue
  $m_0^2 +m_0(N-2s)$, such that
\begin{equation*}
  \la^{-m_0} u(\la x+x_0) \to |x|^{m_0}\widehat Y\left(\frac{x}{|x|},0\right)
  \quad \text{as }\la \to 0^+
  \quad \text{in } H^s(B_1'),
\end{equation*}
where $B'_1:=B_1'(0)$ has been defined in \eqref{Brx}, $u$ is
  trivially extended to zero outside $\Omega$ as in 
  \eqref{eq:triv-ext}, and
\begin{equation}\label{eq:asy-u}
  \widehat Y(\theta',\theta_N,\theta_{N+1})=
  \begin{cases}
     Y(\theta',\theta_N,\theta_{N+1}),&\text{if }\theta_N<0,\\
     0,&\text{if }\theta_N\geq0.
  \end{cases}
\end{equation}
\end{theorem}
Unlike the analogous result for the restricted fractional
Laplacian established in \cite{de2021strong}, the order of homogeneity
of limit profiles does not depend on $s$ and it is always an
integer. This is a consequence of the regularity of the eigenfunctions
of \eqref{prob-eigenvalues}, see Appendix
\ref{appendix-eigenvalues-half-sphere} for further details. In
  particular, the eigenfunctions of \eqref{prob-eigenvalues}, after an
  even reflection through the equator $\theta_{N+1}=0$, turn out to be
  smooth thanks to \cite[Theorem 1.1]{STV}; therefore they are much
  more regular than the solutions of the corresponding problem on the
  half-sphere appearing in \cite{de2021strong} and presenting mixed
  boundary conditions which are responsible for a lower regularity.

Theorem \ref{theor-blow-up-down} is proved by passing to the
  trace in the following blow-up result for solutions of the
  extended problem \eqref{prob-extension}. 
\begin{theorem}\label{theor-blow-up-extended}
  Let $N>2s$ and $\Omega\subset \R^N$ be a bounded Lipschitz
  domain. Let $x_0 \in \partial \Omega$ and assume that there exist
  $R>0$ and a function $g$ satisfying \eqref{hypo-g},
  \eqref{hypo-g-boundary-Omega}, and \eqref{hypo-g-Omega}. Let $U$ be
  a non trivial solution to \eqref{prob-extension} in the sense of
  \eqref{eq-weak}, with $h$ satisfying \eqref{hypo-h}.  Then there
  exist $m_0\in \mathbb{N} \setminus \{0\}$ (which is odd
    in the case $N=1$) and eigenfunction $Y$ of
  \eqref{prob-eigenvalues}, associated to the eigenvalue
  $m_0^2 +m_0(N-2s)$, such that, letting $z_0=(x_0,0)$,
\begin{equation}\label{eq:asyU}
  \la^{-m_0} U(\la z+z_0) \to |z|^{m_0}\widehat Y\left(\frac{z}{|z|}\right)
 \quad \text{as }\la \to 0^+
  \quad \text{in } H^1(B_1^+,t^{1-2s}),
\end{equation}
where $B_1^+=\{z=(x,t)\in\R^N\times(0,+\infty):|z|<1\}$ and $U$ is
  trivially extended to zero outside $\mathcal C_\Omega$.
\end{theorem}
In Theorem \ref{theor-blow-up-extended-all} a more precise
  characterization of the function $\widehat Y$ appearing in
  \eqref{eq:asy-u} and \eqref{eq:asyU} is given, by writing it as a linear combination of
  the eigenfunctions $Y_{m_0,k}$ with coefficients computed in \eqref{beta-k}.

From Remark \ref{rem-Y-not-0}, Theorem \ref{theor-blow-up-down} and
Theorem \ref{theor-blow-up-extended} we deduce the following unique
continuation principles.
\begin{corollary}
  Let $N>2s$ and $\Omega\subset \R^N$ be a bounded Lipschitz
  domain. Let $x_0 \in \partial \Omega$ and assume that there exist
  $R>0$ and a function $g$ satisfying \eqref{hypo-g},
  \eqref{hypo-g-boundary-Omega}, and \eqref{hypo-g-Omega}.
 Let $u$ be
  a solution to \eqref{eq-spec-lapla} in the sense of
  \eqref{eq-spec-lapla-weak} and $U$ be
  a solution to \eqref{prob-extension} in the sense of
  \eqref{eq-weak}, with $h$ satisfying \eqref{hypo-h}.
\begin{itemize}
\item[(i)] If $u(x)=O\left(|x-x_0|^k\right)$ as $x \to x_0$ for any
  $k \in \mathbb{N}$, then $u\equiv 0$ in $\Omega$.
\item[(ii)] If $U(z)=O\left(|z-(x_0,0)|^k\right)$ as $z \to (x_0,0)$
  for any $k \in \mathbb{N}$, then $U\equiv 0$ on
  $\mathcal{C}_\Omega$.
\end{itemize}
\end{corollary}
The paper is organized as follows. In Section
\ref{sec-notation-preliminaries} we fix some notation used throughout
the paper and recall some preliminary results concerning
  functional inequalities and trace operators.  In Section
\ref{sec-straightening the boundary} we apply the local diffeomorphism
introduced in \cite{adolfsson1997c1}, see also \cite[Section
2]{de2021strong}, to write an equivalent formulation of problem
  \eqref{prob-extension} on a domain with a straightened lateral
  boundary in a neighbourhood of $x_0$, see
  \eqref{prob-extension-B-}. In Section
\ref{sec-the-monotonicity-formula} we study the Almgren-type
  frequency function associated to the auxiliary problem
  \eqref{prob-extension-B-} and prove its boundedness, on which the
blow-up analysis developed in Section \ref{sec-blow-up-analysis} is
based. Finally in Section \ref{sec-proofs-of-the-main-results} we
prove our main results and in Appendix
\ref{appendix-eigenvalues-half-sphere} we compute the eigenvalues of
problem \eqref{prob-eigenvalues}.

\section{Notations and
  preliminaries} \label{sec-notation-preliminaries} In this section we
present some notation used throughout the paper and prove some
preliminary results concerning functional inequalities and trace
operators.

For every $r>0$, let 
\begin{align*}
  &B_r^+:=\{z\in \R^{N+1}_+: |z| <r\},
    &&S_r^+:=\{z\in \R^{N+1}_+: |z|=r\},  \\
  &B_r':=\{x  \in \R^N: |x| <r\}, 
    &&S_r':=\{x  \in \R^N: |x| =r\}. 
\end{align*}
For every $r>0$ we define the space 
\begin{equation*}
  H^1_{0,S^+_r}(B^+_{r},t^{1-2s}):=
  \overline{\{\phi \in  C^{\infty}(\overline{B^+_r}):\phi
    =0 \text{ in a neighbourhood of } S_r^+\}}^{\norm{\cdot}_{H^1(B^+_{r},t^{1-2s})}},
\end{equation*}
as the closure in $H^1(B^+_{r},t^{1-2s})$ of the set of all functions
in $C^{\infty}(\overline{B^+_r})$ vanishing in a neighbourhood of
$S_r^+$.
\begin{remark}\label{rem:ext-cil} Since
  $B_r^+ \subset B_r' \times (0,+\infty)$, the trivial extension to
  $0$ is a linear and continuous operator from
    $H^1_{0,S^+_r}(B^+_{r},t^{1-2s})$ to
    $H^1_{0,L}(\mathcal{C}_{B'_r},t^{1-2s})$.
\end{remark}

\begin{proposition} \label{prop-traces}
For every $r >0$ there exists a linear and continuous trace operator 
\begin{equation*}
\Tr: H^1(B^+_{r},t^{1-2s})  \to H^s(B_r')\\	
\end{equation*}
such that the restriction of $\Tr$ to
$H^1_{0,S^+_r}(B^+_{r},t^{1-2s})$ coincides with the restriction of
$\mathop{\rm{Tr}}_{B_r'}$ to $H^1_{0,S^+_r}(B^+_{r},t^{1-2s})$.  In
particular, for every $r>0$,
\begin{equation*}
\Tr(H^1_{0,S^+_r}(B^+_{r},t^{1-2s})) \subseteq \mathbb{H}^s(B_r').
\end{equation*}
\end{proposition}
\begin{proof}
  Thanks to Remark \ref{rem:ext-cil}, the operator
  $\mathop{\rm Tr}_{B_r'}$ defined in \eqref{trace-Omega} is well
  defined on $H^1_{0,S^+_r}(B^+_{r},t^{1-2s})$ and
  $\mathop{\rm Tr}_{B_r'}(H^1_{0,S^+_r}(B^+_{r},t^{1-2s})) \subseteq
  \mathbb{H}^s(B_r')$.  Furthermore, as observed in
  \cite[Proposition 2.1]{YYLi} and \cite{MR3023003,MR0350177}, there
  exists a linear, continuous trace operator
  $\Tr: H^1(B^+_{r},t^{1-2s}) \to H^s(B_r')$.  For every
  $u \in \{\phi \in C^{\infty}(\overline{B_r^+}):\phi=0 \text{ on a
    neighbourhood of }
  S_r^+\}$, we have that
  $\Tr(u)=u_{|_{B_r' \times
      \{0\}}}=\mathop{\rm Tr}_{B_r'}(u)$. By density we conclude that $\Tr$ and
  $\mathop{\rm Tr}_{B_r'}$ are equal on
  $H^1_{0,S^+_r}(B^+_{r},t^{1-2s})$.
\end{proof}
We observe that $H^1(B_r^+,t^{1-2s})\subset W^{1,1}(B_r^+)$,
  hence, denoting as $\Tr_1$ the classical trace operator from
  $W^{1,1}(B_r^+)$ to $L^1(S_r^+)$, we can consider its restriction to
  $H^1(B^+_{r},t^{1-2s})$, still denoted as $\Tr_1$; from
  \cite[Theorem 19.7]{MR1069756} and the Divergence Theorem one can
  easily deduce that, for any $r>0$, such a restriction is a linear,
  continuous trace operator
\begin{equation}\label{trace-Sr+}
\mathop{\rm{Tr}_1}: H^1(B^+_{r},t^{1-2s})  \to L^2(S_r^+, t^{1-2s})
\end{equation}
which is also compact.  With a slight abuse of notation, from now on we
will simply write $v$ instead of $\mathop{\rm{Tr}_1}(v)$ on $S^+_r$.

We recall from \cite[Lemma 2.6]{Fall-Felli-2014} the following
Sobolev-type inequality with boundary terms.
\begin{proposition} 
  There exists a constant $\mathcal{S}_{N,s}>0$ such that, for all
  $r>0$ and $v \in H^1(B_r^+,t^{1-2s})$,
\begin{equation}\label{ineq-soblev-trace}
  \left(\int_{B'_r} |\Tr(v)|^{2^*_s} \, dx\right)^{\frac{2}{2^*_s}}
  \le \mathcal{S}_{N,s}\left(\int_{B^+_r} t^{1-2s}|\nabla v|^2 \, dz
    +\frac{N-2s}{2r}\int_{S_r^+} t^{1-2s}v^2 \, dS \right),
\end{equation}
where $2^*_s$ is defined as in \eqref{2ast}.
\end{proposition}
The following inequality will be used to obtain estimates on the
Almgren frequency function.
\begin{proposition}
  Let $\omega_N$ be the $N$-dimensional Lebesgue measure of the unit
  ball in $\R^N$. For any $r >0$, $v \in H^1(B_r^+,t^{1-2s})$ and
  $f \in L^{\frac{N}{2s}+\e}(B_r')$ with $\e >0$, we have
  that
  \begin{equation}\label{ineq-found}
    \int_{B_r'} f |\Tr(v)|^2 \,
    dx \le \eta_f(r)\left(\int_{B_r^+}t^{1-2s} |\nabla v|^2\,
      dz+\frac{N-2s}{2r}\int_{S_r^+} t^{1-2s}v^2 \, dS\right),
\end{equation}
where 
\begin{equation}\label{eta-f}
  \eta_f(r):=\mathcal S_{N,s}\omega_N^{\frac{4 s^2\e}{N(N+2s\e)}}
  \norm{f}_{L^{\frac{N}{2s}+\e}(B_r')}r^{\frac{4s^2\e }{N+2s\e}}.
\end{equation}
\end{proposition}
\begin{proof}
By the H\"older inequality 
\begin{equation*}
\int_{B_r'} f |\Tr(v)|^2 \, dx \le\norm{\Tr(v)}_{L^{2^*_s}(B_r')}^2\norm{f}_{L^{\frac{N}{2s}+\e}(B_r')}\omega_N^\frac{4s^2\e }{N(N+2s\e)} r^{\frac{4s^2\e }{N+2s\e}}.
\end{equation*}
Then \eqref{ineq-found} follows from \eqref{ineq-soblev-trace}.  
\end{proof}

We also recall the following Hardy-type inequality with boundary terms
 from \cite[Lemma 2.4]{Fall-Felli-2014}.
\begin{proposition}	\label{prop-Hardy}
For any $r >0$ and any $v \in H^1(B_r^+,t^{1-2s})$
\begin{equation}\label{ineq-Hardy}
\left(\frac{N-2s}{2}\right)^2\int_{B_r^+} t^{1-2s} \frac{|v(z)|^2}{|z|^2} \, dz \le 
\int_{B_r^+} t^{1-2s}\left(\nabla v\cdot \frac{z}{|z|}\right)^2 \, dz
+\left(\frac{N-2s}{2r}\right)\int_{S_r^+}t^{1-2s} v^2 \, dS.
\end{equation} 
\end{proposition}
The following Poincar\'{e}-type
  inequality directly follows from \eqref{ineq-Hardy}:
for all $r >0$ and $v \in 
H^1(B_r^+,t^{1-2s})$ 
\begin{equation}\label{eq-Poincare}
\int_{B_r^+} t^{1-2s} v^2 \, dz \le \frac{4r}{(N-2s)^2}\left(r\int_{B_r^+} t^{1-2s} |\nabla v|^2 \, dz +\frac{N-2s}2\int_{S_r^+}t^{1-2s} v^2 \, dS \right).	
\end{equation}

\begin{remark}\label{remarknormaequiv}
  As a consequence of \eqref{eq-Poincare}
  and by continuity of the trace operator \eqref{trace-Sr+}, for every
  $r>0$ we have that 
\begin{equation*}
  \left(\int_{S_r^+}t^{1-2s} v^2 \, dS  +\int_{B_r^+} t^{1-2s}
    |\nabla v|^2 \, dz\right)^{1/2}
\end{equation*}
is an equivalent norm on $H^1(B_r^+,t^{1-2s})$.
\end{remark}

\section{Straightening the boundary}\label{sec-straightening the boundary}
Let $x_0\in\partial\Omega$, $R>0$ and $g$ satisfy \eqref{hypo-g},
\eqref{hypo-g-boundary-Omega}, and \eqref{hypo-g-Omega}. Up to a
suitable choice of the coordinate system, it is not restrictive to
assume that
\begin{equation*}
x_0=0, \quad g(0)=0, \quad \nabla g(0)=0.
\end{equation*} 
We use the local diffeomorphism $F$ constructed in \cite[Section
2]{de2021strong} (see also \cite{adolfsson1997c1}) to straighten the
boundary of $\mathcal{C}_\Omega$ in a neighbourhood of $0$; for the
sake of clarity and completeness we summarize its properties in
Propositions \ref{diffeomorphism} and \ref{beta-properties} below,
referring to \cite[Section 2]{de2021strong} for their proofs. We
consider the variable $z=(y,t) \in \R^N \times[0,\infty)$ with
$y=(y',y_N)=(y_1,\cdots,y_N)$. For future reference we define
\begin{equation}\label{M-N}
  M_N:=
  \left(\begin{array}{c|c|c}\mathop{\rm Id}_{N-1} &0&0 \\
          \hline 0&-1&0\\\hline 0&0&1 \end{array}\right),\qquad 
  M_N':=
  \left(\begin{array}{c|c}\mathop{\rm Id}_{N-1} &0 \\
          \hline 0&-1 \end{array}\right),
    \end{equation}
where $\mathop{\rm Id}_{N-1}$ is the identity $(N-1)\times(N-1)$
matrix.

\begin{proposition}\cite[Section
  2]{de2021strong} \label{diffeomorphism} There exist
  $F=(F_1,\dots,F_{N+1})\in C^{1,1}(\R^{N+1}, \R^{N+1})$ and $r_0>0$
  such that $F\big|_{B_{r_0}}:B_{r_0} \to F(B_{r_0})$ is a
  diffeomorphism of class $C^{1,1}$,
\begin{align*}
  &F(y',0,0)=(y',g(y'),0) \quad \text{for all } y' \in \R^{N-1},\\
  &F_N(y',y_N,t)=y_N+g(y') \quad
    \text{for all } (y',y_N,t) \in \R^{N-1}\times \R\times \R,\\
  &F_{N+1}( y,t)=t, \quad \text{for all } (y,t)
    \in \R^N\times \R,\\
  &\alpha(y,t):=\mathop{\rm det}J_F(y,t)>0\quad\text{in }B_{r_0},
\end{align*}
and
\begin{align}\label{F-propeties-4}
  &F(\{(y',y_N,t)\in B_{r_0}^+:y_N=0\})=\partial_L\mathcal{C}_{\Omega}
  \cap F(B_{r_0}^+),\\
  &F(\{(y',y_N,t)\in B_{r_0}^+:y_N<0\})=\mathcal C_\Omega\cap F(B_{r_0}^+),
    \label{eq:4}
\end{align}
where $\partial_L\mathcal{C}_{\Omega}$ is defined in \eqref{C+} and
$J_F(y,t)$ is the Jacobian matrix of $F$.
Furthermore the following properties hold:
\begin{enumerate}[\rm i)]\setlength\itemsep{10pt}
\item $J_F$ depends only on the variable $y$ and
\begin{equation*}
J_F(y',y_{N})=J_F(y)={\mathop{\rm Id}}_{N+1}+O(|y|) \quad
\text{as } |y|\to 0^+,
\end{equation*}
where $\mathop{\rm Id}_{{N+1}}$ denotes the identity
$(N+1)\times (N+1)$ matrix and $O(|y|)$ denotes a matrix with all
entries being $O(|y|)$ as $|y|\to 0^+$;
\item $\alpha(y)=\det{ J_F}(y) =1+O(|y'|^2)+O(y_{N})$ as
  $|y'| \to 0^+$ and $y_{N}\to0$;
\item $\pd{F_i}{t}=\pd{F_{N+1}}{y_i}=0$ for any $i=1,\dots, N$ and
  $\pd{F_{N+1}}{t}=1$.
\end{enumerate}
\end{proposition}

\smallskip\noindent For every $r>0$, let  
\begin{equation}\label{eq:calQ}
\mathcal Q_r:= \{(y',y_N,t)\in B_r^+:y_N<0\},
\end{equation}
so that $F(\mathcal Q_{r_0})=\mathcal C_\Omega\cap F(B_{r_0}^+)$ in
view of \eqref{eq:4}. If $U \in H_{0,L}^1(\mathcal{C}_\Omega,t^{1-2s})$
solves
\eqref{prob-extension}, then the function
\begin{equation}\label{eq:1defW}
  W=U \circ F\in H^1(\mathcal Q_{r_0},t^{1-2s})
\end{equation}
is a weak solution to
\begin{equation}\label{prob-extension-B-}
\begin{cases}
\dive(t^{1-2s}A \nabla W)=0, &\text{in } \mathcal Q_{r_0}, \\
-\lim_{t \to 0^+} t^{1-2s}\alpha \pd{W}{t}= \kappa_{s,N}\bar{h}W,
& \text{on } \mathcal Q'_{r_0},
\end{cases}	
\end{equation}
where $\mathcal Q'_{r}:= \{(y',y_N)\in B'_r :y_N<0\}$ for all $r>0$,
$A=A(y)$ is the $(N+1)\times(N+1)$ matrix-valued function given
by 
\begin{equation*}
  A(y):= (J_F(y))^{-1}(J_F(y)^{-1})^T |\mathrm{det}J_F(y)|,
\end{equation*}
and
\begin{equation}\label{eq:df-h-bar}
  \bar{h}(y)= \alpha(y) h(F(y,0)).
\end{equation}
As observed in \cite[Section 2]{de2021strong}, $A$ has $C^{0,1}$
entries $\big(a_{ij}\big)_{i,j=1}^{N+1}$ and can be written as
\begin{equation}
\label{matrix-A}
A(y)=A(y',y_N)=
\left(\begin{array}{c|c}D(y',y_N) &0 \\ \hline 0&\alpha(y',y_N)\end{array}\right),
\end{equation}
with
\begin{equation}\label{matrix-D}
D(y',y_{N})=
\left(\begin{array}{c|c}\mathop{\rm Id}_{N-1}+O(|y'|^2)+O(y_{N})
        &O(y_{N})
        \\ \hline O(y_{N})&1+O(|y'|^2)+O(y_{N})\end{array}\right),
\end{equation}
where $\mathop{\rm Id}_{N-1}$ is the identity $(N-1) \times (N-1)$
matrix, $O(y_{N})$ and $O(|y'|^2)$ denote blocks of matrices with all
elements being $O(y_{N})$ as $y_{N}\to 0$ and $O(|y'|^2)$ as
$|y'|\to 0$ respectively. In particular, in view of
  \eqref{matrix-A}-\eqref{matrix-D} we have that
  \begin{equation}\label{eq:entries-jN}
    a_{Nj}(y',0)=a_{jN}(y',0)=0\quad\text{for all }j=1,\dots,N-1.
  \end{equation}
Having in mind to reflect our problem through the hyperplane $y_N=0$,
we define
\begin{align}\label{eq:defA}
&\widetilde A(y',y_N):= 
\begin{cases}
A(y',y_N), &\text{if } y_N \le 0, \\
M_N A(y',-y_N) M_N, &\text{if } y_N > 0,
\end{cases}\\
  \label{eq:defD}
&\widetilde D(y',y_N):= 
\begin{cases}
D(y',y_N), &\text{if } y_N \le 0, \\
M_N' D(y',-y_N) M_N', &\text{if } y_N > 0,
\end{cases}
\end{align}
with $M_N,M_N'$ as in \eqref{M-N}, and
\begin{equation}\label{alpha}
\widetilde \alpha(y',y_N):=
\begin{cases}
\alpha(y',y_N), &\text{if } y_N\le 0,\\
\alpha(y',-y_N), &\text{if } y_N> 0,
\end{cases} 
\end{equation}
where $\alpha(y)=\det{ J_F}(y)$. We observe that the Lipschitz
  continuity of $A$ and \eqref{eq:entries-jN} imply that
  the entries of $\widetilde A$ are of class $C^{0,1}$. Furthermore,
$\widetilde A$ is symmetric and, possibly choosing $r_0$ smaller from
the beginning,
\begin{equation}
  \|\widetilde A(y)\|_{\mathcal{L}(\R^{N+1},\R^{N+1})} \le 2
  \quad \text{and}\quad \frac12 |z|^2 \le \widetilde A(y)z \cdot z
  \le 2 |z|^2 \quad \text{for all } z \in \R^{N+1},
  \ y \in \overline{B'_{r_0}}, \label{ellipticity}
\end{equation}
where $\norm{\cdot}_{\mathcal{L}(\R^{N+1},\R^{N+1})}$ denotes the
operator norm  on the space of bounded linear operators from
$\R^{N+1}$ into itself.
We also observe that \eqref{matrix-A}-\eqref{matrix-D} imply the expansion
\begin{equation}\label{A-O}
  \widetilde A(y)=\mathop{\rm Id}\nolimits_{N+1}+O(|y|) \quad \text{as } |y| \to 0^+.
\end{equation}
Letting $\widetilde A$ and $\widetilde D$ be as in
\eqref{eq:defA}-\eqref{eq:defD}, we define
\begin{equation}\label{mu-beta}
  \mu(z):=\frac{\widetilde A(y)z \cdot z}{|z|^2} \quad\text{and}\quad
  \beta(z):=\frac{\widetilde A(y)z }{\mu(z)}
\quad\text{for every $z=(y,t)\in
  \overline{B^+_{r_0}}\setminus \{0\}$},
\end{equation}
and
\begin{equation}\label{beta'}
\beta'(y):=\frac{\widetilde D(y)y}{\mu(y,0)}\quad\text{for every $y\in
  \overline{B_{r_0}'}$}.
  \end{equation}
For every $z=(z_1,\dots,z_{N+1}) \in \R^{N+1}$ and
$y \in \overline{B'_{r_0}}$, $d\widetilde A(y)zz$ is defined  as the vector of
$\R^{N+1}$
with $i$-th component given by
\begin{equation}\label{dA}
  (d\widetilde A(y)zz)_i=\sum_{h,k=1}^{N+1}\pd{\widetilde a_{kh}}{z_i}(y)z_h z_k,
  \quad i=1,\cdots,N+1,
\end{equation}
where $(\widetilde a_{k,h})_{k,h=1}^{N+1}$ are the entries of the matrix $\widetilde A=$ in \eqref{eq:defA}.

\begin{proposition} \label{beta-properties}
Let $\mu$, $\beta$, and $\beta'$ be as in \eqref{mu-beta}-\eqref{beta'}.
Then, possibly choosing $r_0$ smaller from the beginning, we have that 
\begin{align}
  &\frac{1}{2}\le\mu(z) \le 2 \quad \text{for any } z\in
    \overline{B^+_{r_0}}\setminus \{0\}, \label{mu-estimates} \\
&\mu(z)=1 +O(|z|),\quad  \nabla \mu(z)=O(1)\quad \text{as } |z| \to 0^+. \label{mu-O}
\end{align} 
Moreover $\beta$ and $\beta'$ are well-defined and
\begin{align}
  &\beta(z)=z +O(|z|^2)=O(|z|) \quad \text{as } |z| \to 0^+,\label{beta-estimate}\\ 
  &J_{\beta}(z)=\widetilde A(y)+O(|z|)=\mathop{\rm Id}\nolimits_{N+1}+O(|z|), \quad
\mathop{\rm div}(\beta)(z)=N+1+O(|z|)\quad
    \text{as }
    |z| \to 0^+, \label{jacob-beta-estimates}\\
  &\beta'(y)=y +O(|y|^2)=O(|y|),\quad \mathop{\rm div}(\beta')(y)=N+O(|y|)  \quad \text{as } |y| \to 0^+,\label{beta'-estimate}.
\end{align}

\end{proposition}
\begin{proof}
  \eqref{mu-estimates} easily follows from \eqref{ellipticity}. We
  refer to \cite[Lemma 2.1]{de2021strong} for the proof of
  \eqref{mu-O}. As a direct consequence, $\beta$ and $\beta'$ are
  well-defined.  From \eqref{beta-estimate} and
  \eqref{jacob-beta-estimates}, whose proof is contained in
  \cite[Lemma 2.2]{de2021strong}, we derive \eqref{beta'-estimate},
  after noting that $\beta'$ coincides with the first
  $N$-components of the vector $\beta$.
\end{proof}

\begin{remark}
 From the Lipschitz continuity of $\widetilde A$ observed above and Proposition
 \ref{beta-properties} we have that
\begin{align}\label{everything-bounded}
  &\widetilde A \in C^{0,1}(\overline {B^+_{r_0}}, \R^{(N+1)^2}),  \
    \mu  \in C^{0,1}(\overline {B^+_{r_0}}), \ 
    \frac{1}{\mu}  \in C^{0,1}(\overline {B^+_{r_0}}), \
    \beta \in   C^{0,1}(\overline {B^+_{r_0}},\R^{N+1}) \\
  &J_\beta \in L^\infty(\overline {B^+_{r_0}}, \R^{(N+1)^2}),  \  \dive(\beta)
    \in
    L^\infty(\overline
    {B^+_{r_0}}),
    \  \beta' \in L^\infty(\overline {B_{r_0}'},\R^N), 
    \   \dive(\beta') \in  L^\infty(\overline {B'_{r_0}}).\notag 
\end{align} 
\end{remark}

\begin{remark}\label{null-hyperplane}
  If $v\in H^1_{0,L}(\mathcal{C}_\Omega ,t^{1-2s})$, then
  $\restr{(v \circ F)}{\mathcal Q_{r_0}} \in H^1(\mathcal
    Q_{r_0},t^{1-2s}) $ by Proposition \ref{diffeomorphism}, and
\begin{equation}\label{trace-hyperplane}
  (v\circ F)(z)=0 \quad \text{ for any } z\in
  \{(y',y_N,t)\in B_{r_0}^+:y_N=0\}
\end{equation}
in view of \eqref{F-propeties-4}.  Equality \eqref{trace-hyperplane} is
meant in the sense of the classical theory of  traces for Sobolev spaces;
this is possible thanks to the
fact that $H^1(E,t^{1-2s}) \subset W^{1,1}(E)$ for any bounded open set
$E \subseteq \R^{N}\times (0,\infty)$.
\end{remark}

If $W$ is a solution to \eqref{prob-extension-B-}, let $\widetilde{W}$
be defined as follows  
\begin{align}
\widetilde{W}(y',y_N,t):=&
\begin{cases}
W (y',y_N,t) ,	&\text{ if } (y',y_N,t) \in \mathcal Q_{r_0},\\ 
-W (y',-y_N,t) ,	&\text{ if } (y',y_N,t) \in B_{r_0}^+\text{ and } y_N >0.
\end{cases} \label{W}
\end{align}
For the sake of convenience we will still denote $\widetilde{W}$ with
$W$.  Letting $\bar{h}$ be defined in \eqref{eq:df-h-bar}, we also
consider the following function
 \begin{equation}\label{tilde-h}
 \widetilde{h}(y',y_N):=\begin{cases}
	\bar{h}(y',y_N) ,	&\text{ if } (y',y_N) \in \mathcal Q_{r_0}',\\ 
	\bar{h} (y',-y_N) ,	&\text{ if } (y',y_N) \in B_{r_0}',\text{ and } y_N >0.
\end{cases}
 \end{equation}
 It is easy to verify that $W \in H^1(B_{r_0}^+,t^{1-2s})$ thanks to
 Remark \ref{null-hyperplane} and
\begin{equation}\label{htildedovesta}
\widetilde{h} \in W^{1, \frac{N}{2s}+\e}(B_{r_0}')
\end{equation}
thanks to \eqref{hypo-h}, \eqref{eq:df-h-bar} and Proposition \ref{diffeomorphism}.
Furthermore $W$ weakly solves 
\begin{equation}\label{prob-extension-straith}
\begin{cases}
  \dive(t^{1-2s} \widetilde A\nabla W)=0, &\text{ on }B_{r_0}^+, \\
  -\lim_{t \to 0^+} t^{1-2s} \widetilde \alpha \pd{W}{t}=
  \kappa_{s,N}\widetilde h \Tr(W), & \text{ on } B_{r_0}',
\end{cases}		
\end{equation} 
with $\widetilde \alpha$  defined in  \eqref{alpha}, $\widetilde{h}$
in \eqref{tilde-h} and $\widetilde  A$ in \eqref{eq:defA}, namely 
\begin{equation}\label{eq-extension-straith}
	\int_{B_{r_0}^+} t^{1-2s} \widetilde A \nabla W \cdot \nabla
        \phi \, dz
        =\kappa_{s,N}\int_{B_{r_0}'}\widetilde h \Tr(W)\Tr(\phi) \, dy
        \quad \text{for all }\phi \in  H_{0,S^+_{r_0}}^1(B_{r_1}^+,t^{1-2s}).
\end{equation}
Thanks to Proposition \ref{prop-traces}, \eqref{htildedovesta} and
the H\"older inequality, the second member of \eqref{eq-extension-straith}
is well-defined.

\begin{remark}\label{rem:regularity}
  In \cite[Theorem 2.1]{fellisiclari} it is proved that, if
  $W \in H^1(B_{r_0}^+,t^{1-2s})$ is a weak solution to
  \eqref{eq-extension-straith} with $\widetilde A$ and $\widetilde h$
  satisfying \eqref{matrix-A}, \eqref{eq:defA},
  \eqref{everything-bounded}, \eqref{mu-estimates},
  \eqref{htildedovesta}, then
  \begin{equation}\label{eq:regularity}
\nabla_x W \in H^1(B_r^+,t^{1-2s})\quad\text{and}\quad
t^{1-2s}\frac{\partial W}{\partial t}\in H^1(B_r^+,t^{2s-1})
  \end{equation}
  for all $r \in (0,r_0)$. Furthermore
  \begin{equation*}
  \|\nabla_x W\|_{H^1(B_{r}^+,t^{1-2s})}+
  \left\|t^{1-2s}\frac{\partial W}{\partial t}\right\|_{H^1(B_{r}^+,t^{2s-1})}
  \leq C
    \norm{W}_{H^1(B^+_{r_0},t^{1-2s})}
  \end{equation*}
for a positive constant $C>0$ depending only on $N$, $s$, $r$, $r_0$,
$\|\widetilde h\|_{W^{1,\frac{N}{2s}}(B'_{r_0})}$, 
$\|\widetilde A\|_{W^{1,\infty}(B_{r_0}^+,\R^{(N+1)^2})}$ (but independent of $W$).  
\end{remark}

\begin{remark}\label{coarea-well-defined}
  If $W \in H^1(B_{r_0}^+,t^{1-2s})$ is a weak solution to
  \eqref{eq-extension-straith}, the regularity result
  \eqref{eq:regularity} and \eqref{trace-Sr+} ensure that, for all
  $\phi\in H^1(B_{r_0}^+,t^{1-2s})$ and 
  $r\in(0,r_0)$,
  $t^{1-2s}\Tr_1(\widetilde D\nabla_xW\cdot x) \Tr_1\phi \in
  L^1(S_r^+)$; moreover the function
  \begin{equation*}
    r\mapsto \int_{S_r^+}t^{1-2s}(\widetilde D\nabla_xW\cdot x) \phi \,dS
  \end{equation*}
  is continuous in $(0,r_0)$. Furthermore, since
  $t^{1-2s}\frac{\partial W}{\partial t}\in H^1(B_r^+,t^{2s-1})$ for
  all $r\in(0,r_0)$ by \eqref{eq:regularity}, we also have that,  
   for all
  $\phi\in H^1(B_{r_0}^+,t^{1-2s})$ and 
  $r\in(0,r_0)$,
  $t^{1-2s}\widetilde \alpha \frac{\partial W}{\partial t} t \phi\in
  W^{1,1}(B_r^+)$, so that $\Tr_1(t^{1-2s}\widetilde \alpha
  \frac{\partial W}{\partial t} t \phi) \in
  L^1(S_r^+)$; moreover the function
  \begin{equation*}
    r\mapsto \int_{S_r^+}t^{1-2s}\widetilde \alpha
    \frac{\partial W}{\partial t} t \phi \,dS
  \end{equation*}
  is continuous in $(0,r_0)$. We conclude that,   for all
  $\phi\in H^1(B_{r_0}^+,t^{1-2s})$, the function 
  \begin{equation*}
    t^{1-2s}   (\widetilde A\nabla W\cdot z)\phi
    =t^{1-2s}(\widetilde D\nabla_xW\cdot x) \phi+
    t^{1-2s}\widetilde \alpha \frac{\partial W}{\partial t} t \phi
  \end{equation*}
  has a trace on $S_r^+$ for all  $r\in(0,r_0)$ and the function
  \begin{equation*}
    r\mapsto \int_{S_r^+}t^{1-2s}
    (\widetilde A\nabla W\cdot z)\phi\,dS 
  \end{equation*}
  is continuous in $(0,r_0)$.  
\end{remark}

The following result provides an integration by parts formula which
will be useful in Section \ref{sec-blow-up-analysis}. 
\begin{proposition}\label{prop-forula-div}
Let $W$ be a weak solution
  to \eqref{prob-extension-straith}. For all $r \in(0,r_0)$ and
    $\phi \in H^1(B_{r_0}^+,t^{1-2s})$
	\begin{equation}\label{formula-div}
          \int_{B_r^+ }t^{1-2s} \widetilde A \nabla W \cdot
          \nabla \phi
          \, dz =\frac{1}{r}\int_{S_r^+} t^{1-2s} (\widetilde A \nabla
          W \cdot z )
          \phi \, dS + \kappa_{s,N}\int_{B_r'}
          \widetilde h \Tr(W) \Tr(\phi) \, dx.
	\end{equation} 
\end{proposition}
\begin{proof}
  By density it is enough to prove \eqref{formula-div} for
  $\phi \in C^\infty (\overline{B_{r_0}^+})$.  Let $r\in(0,r_0)$. For
  every $n\in{\mathbb N}$, let
\begin{equation*}
\eta_n(z):= 
\begin{cases}
1, &\text{ if } \quad 0 \le |z| \le r-\frac{1}{n}, \\
n(r-|z|),   &\text{ if } \quad  r-\frac{1}{n}\le |z| \le  r,\\
0, &\text{ if }    \quad |z| \ge r.
\end{cases}
\end{equation*}
Testing \eqref{eq-extension-straith} with $\phi\eta_n$ and passing
to the limit as $n \to \infty$, we obtain \eqref{formula-div} thanks to
the integral mean value theorem and
Remark \ref{coarea-well-defined}.
\end{proof}

\begin{remark}
  For all $r \in (0, r_0]$ and any $v \in H^1(B_r^+,t^{1-2s})$, thanks
  to \eqref{ineq-found}, \eqref{ellipticity} and \eqref{mu-estimates},
\begin{multline*}
  \int_{B_r^+} t^{1-2s} |\nabla v|^2 \, dz \le 2\int_{B_r^+} t^{1-2s}
  \widetilde A\nabla v\cdot \nabla v \, dz -2 \kappa_{N,s}\int_{B_r'}
  \widetilde h |\Tr(v)|^2 \, dx\\
  +2\kappa_{N,s}\eta_{\tilde{h}}(r)\left(\int_{B_r^+}t^{1-2s} |\nabla
    v|^2\, dz +\frac{N-2s}{r}\int_{S_r^+} t^{1-2s}\mu v^2 \,
    dS\right).
\end{multline*}
Therefore, if $\eta_{\tilde{h}}(r) < \frac{1}{2\kappa_{N,s}}$,
\begin{multline}\label{ineq-nabla-E}
  \int_{B_r^+} t^{1-2s} |\nabla v|^2 \, dz \le
  \frac{2}{1-2\kappa_{N,s}\eta_{\tilde{h}}(r)}\left(\int_{B_r^+}
    t^{1-2s} \widetilde A\nabla v\cdot \nabla v \, dz - \kappa_{N,s}\int_{B_r'}
    \widetilde h |\Tr(v)|^2 \, dx\right) \\
  +\frac{2(N-2s)\kappa_{N,s}
    \eta_{\tilde{h}}(r)}{(1-2\kappa_{N,s}\eta_{\tilde{h}}(r))r}
  \int_{S_r^+} t^{1-2s}\mu v^2 \, dS.
\end{multline}
\end{remark}

\section{The Monotonicity Formula} \label{sec-the-monotonicity-formula}

Let $W$ be a non trivial weak solution of
\eqref{prob-extension-straith}. For any $r \in (0,r_0]$ we define the
height function and the energy function as
\begin{align}
  &H(r):=\frac{1}{r^{N+1-2s}}\int_{S_r^+} t^{1-2s}\mu W^2 \, dS,\label{H} \\
  &D(r):=\frac{1}{r^{N-2s}}\left(\int_{B_r^+} t^{1-2s}\widetilde A
    \nabla W
    \cdot \nabla W \, dz -\kappa_{N,s}
    \int_{B_r'} \widetilde h |\Tr W|^2 \, dx \right), \label{D}
\end{align}
respectively. Eventually choosing $r_0$ smaller from the beginning, we
may assume that
\begin{equation}\label{etapiccolo}
  \eta_{\tilde{h}}(r) < \frac{1}{4\kappa_{N,s}}
  \quad \text{for all $r \in (0,r_0]$},
\end{equation}
so that \eqref{ineq-nabla-E} holds for every $r \in (0,r_0]$.

\begin{proposition}\label{derivataH}
  Let $H$ and $D$ be as in \eqref{H} and \eqref{D}. Then
  $H \in W^{1,1}_{\rm loc}((0,r_0])$ and
\begin{equation}\label{H'1}
  H'(r)=\frac{2}{r^{N+1-2s}} \int_{S_r^+} t^{1-2s}\mu W \pd{W}{\nu}\,dS
  +H(r)O(1) \quad \text{  as } r \to 0^+
\end{equation}
in the sense of distributions and almost everywhere, where $\nu$ is
the outer normal vector to $B_r^+$ on $S_r^+$, i.e.
$ \nu(z):=\frac{z}{|z|}$.  Moreover, we have that almost everywhere
\begin{equation}\label{H'2}
  H'(r)=\frac{2}{r^{N+1-2s}}\int_{S_r^+} t^{1-2s}(\widetilde A\nabla W \cdot\nu)
  W \, dS +H(r)O(1) \quad \text{  as } r \to 0^+	
\end{equation}
and 
\begin{equation}\label{D-as-H'}
H'(r)=\frac{2}{r}D(r)+H(r)O(1)	 \quad \text{ as } r \to 0^+. 
\end{equation}
\end{proposition}
\begin{proof}
  The proof is similar to that of \cite[Lemma 3.1]{de2021strong}
  thus we omit it.
\end{proof}

\begin{proposition}\label{H>0}
We have that $H(r) >0$ for every $r \in (0,r_0]$.
\end{proposition}
\begin{proof}
  Let us assume by contradiction that there exists $r \in (0,r_0]$
  such that $H(r)=0$. Then, from \eqref{H} and \eqref{mu-estimates} we
  deduce that $W \equiv 0$ on $S_r^+$. Thus we can test
  \eqref{eq-extension-straith} with $W$, obtaining that
\begin{align*}
  0&=\int_{B_r^+} t^{1-2s} \widetilde A \nabla W \cdot \nabla W \, dz
    -\kappa_{N,s}\int_{B_r'} \widetilde h |\Tr(W)|^2 \, dx \\
  & \ge \left(\frac{1}{2}-
    \kappa_{N,s}\eta_{\tilde{h}}(r)\right)\norm{\nabla W}^2_{L^2(B^+_r,t^{1-2s})},
\end{align*}
thanks to \eqref{ineq-nabla-E}. Then, by
\eqref{etapiccolo} we can conclude that $W \equiv 0$ on $B_r^+$; this
implies that $W \equiv 0$ on $B^+_{r_0}$ by classical unique
continuation principles for second order elliptic operators with
Lipschitz coefficients (see e.g. 
\cite{MR833393}), giving rise to a contradiction.
\end{proof}
The following proposition contains a Pohozaev-type identity for
problem \eqref{prob-extension-straith}. For its proof we refer to
  \cite[Proposition 2.3]{fellisiclari}, where a more general version
  is established  exploiting some
Sobolev-type regularity results.
\begin{proposition}\cite[Proposition 2.3]{fellisiclari}
  Let $W$ be a weak solution to equation
  \eqref{prob-extension-straith}. Then, for a.e. $r \in (0, r_0)$,
  \begin{align} \label{eq-Poho-identity}
    &\int_{S_r^+} t^{1-2s}\widetilde A \nabla W \cdot \nabla W \, dS-
      \kappa_{N,s}\int_{S_r'} \widetilde h|\Tr(W)|^2 \, dS'\\
    &=2\int_{S_r^+}t^{1-2s}\frac{| \widetilde A \nabla W \cdot \nu|^2}
      {\mu}\,dS
      -\frac{\kappa_{N,s}}{r}\int_{B_r'}(\mathop{\rm{div_y}}(\beta')
      \widetilde h
      +\beta'\cdot \nabla \widetilde h)|\Tr(W)|^2 \, dy \notag \\
    &+\frac{1}{r}\int_{B_r^+}t^{1-2s} \widetilde A \nabla W \cdot
      \nabla W \dive(\beta) \, dz -\frac{2}{r}
      \int_{B_r^+}t^{1-2s}J_\beta(\widetilde A \nabla W)
      \cdot \nabla W \, dz \notag\\
    &+\frac{1}{r}\int_{B_r^+}t^{1-2s} (d\widetilde A \,\nabla W \,\nabla W)
      \cdot\beta\, dz +\frac{1-2s}{r}\int_{B_r^+}t^{1-2s}
      \frac{\widetilde \alpha}{\mu} \widetilde  A \nabla W
      \cdot \nabla W \, dz\notag,
\end{align}
where $\mu$ and $ \beta $ are defined in \eqref{mu-beta},
$\widetilde\alpha$ in \eqref{alpha}, $\beta'$ in \eqref{beta'}, $\nu$
is the outer normal vector to $B_r^+$ on $S_r^+$, i.e.
$\nu(z)=\frac{z}{|z|}$, and $dS'$ denotes the volume element on
  $(N-1)$-dimensional spheres.
\end{proposition}

\begin{remark}
As in Remark \ref{coarea-well-defined}, by the
  Coarea Formula  we have that
  \begin{equation*}
    \int_{B_{r_0}'} |\widetilde h| |\Tr(W)|^2 \, dx
  =\int_0^{r_0}\left(\int_{S_\rho'} |\widetilde h| |\Tr(W)|^2 \,
    dS'\right)
  \,d\rho ,
\end{equation*}
$S'_\rho$ era già stata definita hence
$\rho \to \int_{S_\rho'} \widetilde h|\Tr(W)|^2 \, dS'$ is a
well-defined $L^1(0,r_0)$-function, as a consequence of
\eqref{htildedovesta}, \eqref{ineq-soblev-trace} and the H\"{o}lder
inequality.
\end{remark}

\begin{proposition}
Let $D$ be as in \eqref{D}. Then $D \in W^{1,1}_{\rm loc}((0,r_0])$ and 
\begin{equation}\label{eq-D'}
  D'(r)=2r^{2s-N}\int_{S_r^+}t^{1-2s}\frac{| \widetilde A \nabla W \cdot \nu|^2}
  {\mu}\,dS + O\left(r^{-1+\frac{4s^2 \e}{N+2s \e}}\right)
  \left[D(r)+\frac{N-2s}{2}H(r)\right] 
\end{equation}
as  $r \to 0^+$,  in the sense of distributions and almost everywhere.
\end{proposition}
\begin{proof}
  By the Coarea Formula $D \in W^{1,1}_{\mathrm{loc}}((0,r_0])$ and
\begin{multline}\label{eq:17}
  D'(r)=	(2s-N)r^{2s-N-1} \left(\int_{B_r^+} t^{1-2s}\widetilde
    A
    \nabla W \cdot \nabla W \, dz -\kappa_{N,s}
    \int_{B_r'} \widetilde h |\Tr (W)|^2 \, dx \right)\\
  +r^{2s-N} \left(\int_{S_r^+} t^{1-2s}\widetilde A
    \nabla W \cdot \nabla W \, dS
    -\kappa_{N,s} \int_{S_r'} \widetilde h |\Tr( W)|^2 \, dS'\right)
\end{multline} 
a.e. and in the sense of distributions in $(0,r_0)$.  Using
\eqref{eq-Poho-identity} to estimate the second term on the right hand
side of \eqref{eq:17}, we have that, for a.e. $r \in (0,r_0)$,
\begin{align}\label{eq:18}
  &D'(r)=(2s-N)r^{2s-N-1} \left(\int_{B_r^+} t^{1-2s}\widetilde A
    \nabla W \cdot \nabla W \, dz -\kappa_{N,s} \int_{B_r'}
    \widetilde h |\Tr (W)|^2 \, dx \right)	\\
  &+r^{2s-N}\left( 2\int_{S_r^+}t^{1-2s}
    \frac{| \widetilde A \nabla W \cdot \nu|^2} {\mu}\,dS
    -\frac{\kappa_{N,s}}{r}\int_{B_r'}(\mathop{\rm{div_y}}(\beta')\widetilde
    h +\beta'\cdot \nabla \widetilde h)|\Tr(W)|^2 \, dy\right)  \notag \\
  &+r^{2s-N}\left(\frac{1}{r}\int_{B_r^+}t^{1-2s} \widetilde A \nabla
    W\cdot \nabla W \dive(\beta) \, dz -\frac{2}{r}
    \int_{B_r^+}t^{1-2s}J_\beta(\widetilde A \nabla W) \cdot \nabla W
    \, dz\right)
    \notag\\
  &+r^{2s-N}\left(\frac{1}{r}\int_{B_r^+}t^{1-2s} (d\widetilde A
    \nabla W \nabla W)
    \cdot \beta \, dz +\frac{1-2s}{r}\int_{B_r^+}t^{1-2s}
    \frac{\widetilde \alpha}{\mu}\widetilde  A \nabla W
    \cdot \nabla W \, dz\right)\notag.
\end{align}
Furthermore, thanks to point ii) of Proposition \ref{diffeomorphism},
\eqref{alpha}, \eqref{ellipticity}, \eqref{mu-estimates},
\eqref{mu-O}, \eqref{beta-estimate}, \eqref{jacob-beta-estimates}, and 
\eqref{ineq-nabla-E}, we deduce that
\begin{align} \label{eq:D'-estimates-1}
  &r^{2s-N-1}\!\!\int_{B_r^+} \!t^{1-2s}\left[\left(2s-N+\dive(\beta)+(1-2s)
    \tfrac{\widetilde\alpha}{\mu}\right) \widetilde A \nabla W \cdot
    \nabla W-2J_\beta(\widetilde  A \nabla W) \cdot \nabla W\right] dz\\
&  + r^{2s-N-1}\int_{B_r^+}t^{1-2s}(d \widetilde A\, \nabla W\, \nabla W)\cdot\beta \, dz = O(r)\,r^{2s-N-1}\int_{B_r^+} t^{1-2s} |\nabla W|^2 \, dz \notag\\
&= O(1)\left[D(r)+\frac{N-2s}{2}H(r)\right] \text{ as } r \to 0^+, \notag
\end{align}
where we used also the fact that
$d \widetilde A\,\nabla W\,\nabla W=O(1)|\nabla W|^2$ as $r \to 0^+$
by \eqref{dA} and \eqref{everything-bounded}.

In addition, recalling that
$\tilde{h}\in W^{1,\frac{N}{2s}+\varepsilon}(B'_{r_1})$, from
\eqref{ineq-found}, \eqref{eta-f}, \eqref{everything-bounded} and
\eqref{ineq-nabla-E} it follows that
\begin{multline}\label{eq:D'-estimates-2}
  r^{2s-N-1}\int_{B_r'}[(2s-N+\mathop{\rm{div_y}}(\beta'))\widetilde h
  +\beta'\cdot\nabla \widetilde h]|\Tr(W)|^2\, dx\\
= O\left(r^{-1+\frac{4s^2 \e}{N+2s\e}}\right)\left[D(r)+\frac{N-2s}{2}H(r)\right]
\end{multline}
as $r\rightarrow 0^+$.  Combining \eqref{eq:18},
\eqref{eq:D'-estimates-1} and \eqref{eq:D'-estimates-2}, we obtain
\eqref{eq-D'}.
\end{proof}
 For every $r \in (0,r_0]$ we define the \emph{frequency function}
\begin{equation}\label{N}
	\mathcal{N}(r):=\frac{D(r)}{H(r)}.
\end{equation}
Definition \eqref{N} is well-posed thanks
to Proposition \ref{H>0}.
\begin{proposition}\label{prop-N'}
We have that  $\mathcal{N} \in W^{1,1}_{\rm loc}((0,r_0])$ and
\begin{equation}\label{ineq-N-below-estimate}
\mathcal{N}(r) >-\frac{N-2s}{2} \quad \text{for every } r \in (0,r_0].
\end{equation}
Furthermore, if $\nu(z):=\frac{z}{|z|}$ is the outer normal
vector to $B_r^+$ on  $S_r^+$ and
\begin{equation*}
  \mathcal{V}(r):=2r\frac{ \left(\int_{S_r^+}t^{1-2s}\mu W^2\,
      dS\right)
    \left(\int_{S_r^+}t^{1-2s}\frac{|A\nabla W \cdot \nu|^2}{\mu}\, dS
    \right)-\left(\int_{S_r^+}t^{1-2s}W A\nabla W \cdot \nu
      \, dS\right)^2}{\left(\int_{S_r^+}t^{1-2s}\mu W^2\, dS\right)^2},
\end{equation*}
then
\begin{equation}\label{ineq-V-positive}
\mathcal{V}(r) \ge 0 \quad \text{ for a.e. }r \in (0,r_0)
\end{equation} and, for a.e. $r\in (0,r_0)$,
\begin{equation}\label{eq-N'-estimate}
  \mathcal{N}'(r)-\mathcal{V}(r)
  =O\left(r^{-1+\frac{4s^2\e}{N+2s\e}}\right)\left[\mathcal{N}(r)
    +\frac{N-2s}{2}\right]\quad \text{as } r \to 0^+.
\end{equation}
\end{proposition}
\begin{proof}
  Since $D \in W^{1,1}_{\rm loc}((0,r_0])$ and
  $\frac{1}{H}\in W^{1,1}_{\rm loc}((0,r_0])$ by Proposition
  \ref{derivataH} and Proposition \ref{H>0}, then
  $\mathcal{N}\in W^{1,1}_{\rm loc}((0,r_0])$. Furthermore we recall
  that \eqref{ineq-nabla-E} holds for every $r\in (0,r_1]$, thus
\begin{equation}\label{liminf-positivo}
\mathcal{N}(r)\ge-\kappa_{N,s}(N-2s) \eta_{\tilde{h}}(r),
\end{equation}
for every $r\in (0,r_0]$ and, in virtue of this,
\eqref{ineq-N-below-estimate} directly follows from
\eqref{etapiccolo}.  Moreover \eqref{ineq-V-positive} is a consequence
of the Cauchy-Schwarz inequality in $L^2(S_r^+,t^{1-2s})$.  From
\eqref{H'2}, \eqref{D-as-H'} and \eqref{eq-D'} we deduce that
\begin{align}\label{eq:N'-estimates-1}
  \mathcal{N}'(r)=&\frac{D'(r)H(r)-D(r)H'(r)}{(H(r))^2}
  =\frac{D'(r)H(r)-\frac{r}{2}(H'(r))^2+O(r)H(r)H'(r)}{(H(r))^2}\\
\notag  =& \mathcal{V}(r)+O(r)+
           O(r^{-1+\frac{4s^2\e}{N+2s\e}})\left[\mathcal{N}(r)
           +\frac{N-2s}{2}\right]\\
 \notag &\ +
  \frac{O(r^{-N+2s})}{H(r)}\int_{S^+_r}t^{1-2s}(A\nabla W\cdot \nu)W\,
  dS
\end{align}
as $r\rightarrow 0^+$.  In order to deal with the last term in
\eqref{eq:N'-estimates-1}, we observe that, for a.e. $r\in (0,r_0)$,
\begin{equation*}
  \int_{S^+_r}t^{1-2s}(A\nabla W\cdot \nu)W\, dS=r^{N-2s}D(r)
  +H(r)O(r^{N+1-2s})\quad \text{as $r\rightarrow 0^+$},
\end{equation*}
in virtue of \eqref{H'2} and \eqref{D-as-H'}. 
Thus, substituting into \eqref{eq:N'-estimates-1}, we conclude that
\begin{equation}
  \mathcal{N}'(r) = \mathcal{V}(r) +
  O(r^{-1+\frac{4s^2\e}{N+2s\e}})\left[\mathcal{N}(r)+\frac{N-2s}{2}
  \right]\quad \text{as $r\rightarrow 0^+$},
\end{equation}
where we have used that $\frac{4s^2\e}{N+2s\e}<1$ since
$\varepsilon\in(0,1)$ and $N>2s$. Estimate
\eqref{eq-N'-estimate} is thereby proved.
\end{proof}

\begin{proposition}
There exists a constant $C >0$ such that, for every $r \in (0,r_0]$, 
\begin{equation}\label{ineq-N-bounded}
\mathcal{N}(r) \le C.
\end{equation}
\end{proposition}
\begin{proof}
  From \eqref{ineq-V-positive} and \eqref{eq-N'-estimate} we deduce
  that there exists a constant $c >0$ such that
\begin{equation}\label{eq119}
  \left(\mathcal{N}(r)+\frac{N-2s}{2}\right)' \ge
  -c\, r^{-1+\frac{4s^2\e}{N+2s\e}}\left(\mathcal{N}(r)
    +\frac{N-2s}{2}\right) \quad \text {for a.e. }r \in (0,r_1),
\end{equation}
for some $r_1\in(0,r_0)$ sufficiently small.  Hence, thanks to
\eqref{ineq-N-below-estimate}, we are allowed to divide each member of
\eqref{eq119} by $\mathcal{N}(r)+\frac{N-2s}{2}$, obtaining that
\begin{equation}
  \left(\log\left(\mathcal{N}(r)+\frac{N-2s}{2}\right)\right)'
  \ge -c\, r^{-1+\frac{4s^2\e}{N+2s\e}}  \quad \text {for a.e. }r \in (0,r_1).
\end{equation}
Then, integrating over $(r,r_1)$ with $r<r_1$, we have that 
\begin{equation}
  \mathcal{N}(r) \le -\frac{N-2s}{2}+
  \exp\left(c\,\frac{N+2s\e}{4s^2\e}r_1^{\frac{4s^2\e}{N+2s\e}}\right)
  \left(\mathcal{N}(r_1)+\frac{N-2s}{2} \right)
  \quad \text {for every }r \in (0,r_1),
\end{equation}
which proves \eqref{ineq-N-bounded}, taking into account the
  continuity of $\mathcal N$ in $(0,r_0]$.
\end{proof}
\begin{proposition}\label{prop-limit-N}
There exists the limit
\begin{equation}\label{limit-N}
\gamma:=\lim_{r \to 0^+} \mathcal{N}(r).
\end{equation}
Moreover $\gamma $ is finite and $\gamma \ge 0$.
\end{proposition}
\begin{proof}
Combining \eqref{ineq-N-bounded} and \eqref{eq119}, we infer that 
\begin{equation}\label{eq6}
  \left(\mathcal{N}(r)+\frac{N-2s}{2}\right)' \ge -
  c\,r^{-1+\frac{4s^2\e}{N+2s\e}}  \left(C+\frac{N-2s}{2}\right) 
\end{equation}
for a.e. $r \in (0,r_1)$, hence
\begin{equation}
  \left(\frac{N-2s}{2}+\mathcal{N}(r)+c\,
    \left(\frac{N-2s}{2}+C\right)\frac{N+2s\e}{4s^2\e}
    r^{\frac{4s^2\e}{N+2s\e}}\right)' \ge 0  \quad
  \text {for a.e. }r \in (0,r_1).
\end{equation}
From this, it follows in particular that the limit $\gamma$ 
in \eqref{limit-N} exists. Moreover, by \eqref{ineq-N-below-estimate}
and \eqref{ineq-N-bounded} $\gamma$ is finite, whereas
\eqref{liminf-positivo} implies that $\gamma \ge 0$.
\end{proof}
\begin{proposition}\label{l:doub}
There exist $c_0,\bar c>0$ and $\bar r\in(0,r_0)$ such that 
\begin{equation}\label{ineq-H-upper-estimate}
H(r)\le c_0\, r^{2 \gamma} \quad \text{for all } r \in (0,r_0]
\end{equation}
and
  \begin{equation}\label{eq:doubling}
    H(Rr)\leq R^{\bar c} \,H(r)\quad \text{for all }
    R\geq 1\text{ and }r \in \big(0,\tfrac{\bar r}R\big].
  \end{equation}
Furthermore, for any $\sigma >0$ there exists a constant $c_\sigma>0$
such that
\begin{equation}\label{ineq-H-lower-estimate}
  H(r)\ge c_\sigma r^{2 \gamma+\sigma} \quad
  \text{for all } r \in (0,r_0].
\end{equation}
\end{proposition}
\begin{proof}
By \eqref{limit-N} we have that $\mathcal N(r)=\gamma+\int_0^r
  \mathcal N'(t)\,dt$; hence from \eqref{D-as-H'} it follows that 
\begin{equation}\label{eq:2}
  \frac{H'(r)}{H(r)}=\frac{2}{r}\mathcal{N}(r)+O(1)=
  \frac{2}{r}\int_0^r \mathcal{N}'(t) \, dt +\frac{2\gamma}{r}+O(1).
\end{equation}
From \eqref{eq6} and up to choosing $r_1$ smaller, it follows that, for
a.e. $r\in (0,r_1)$,
\begin{equation*}
\frac{H'(r)}{H(r)}\ge -\kappa r^{-1+\frac{4s^2\e}{N+2s\e}}+\frac{2\gamma}{r}
\end{equation*}
for some positive constant $\kappa>0$. 
Then an integration over $(r,r_1)$ yields
\begin{equation*}
  \log\left(\frac{H(r_1)}{H(r)}\right) \ge-
  \kappa\frac{N+2s\e}{4s^2\e} \left(r_1^{\frac{4s^2\e}{N+2s\e}}
    -r^{\frac{4s^2\e}{N+2s\e}}\right)+ \log\left(\frac{r_1}{r}
  \right)^{2\gamma}
\end{equation*}
and thus
\begin{equation*}
  H(r) \le \frac{H(r_1)}{r_1^{2 \gamma}}
  \exp\left(\kappa\frac{N+2s\e}{4s^2\e}
    r_1^{\frac{4s^2\e}{N+2s\e}}\right) r^{2 \gamma}
\end{equation*}
for all $r \in (0,r_1]$, thus implying
  \eqref{ineq-H-upper-estimate} thanks to the continuity of $H$ in
  $(0,r_0]$.

To prove \eqref{eq:doubling}, we observe that \eqref{eq:2} and
  \eqref{ineq-N-bounded} imply that, for some $\bar r\in(0,r_0)$ and
  $\bar c>0$,
\begin{equation}
  \frac{H'(r)}{H(r)}\leq \frac{\bar c}{r}
  \quad\text{for all }r\in(0,\bar r), 
\end{equation}
whose integration over $(r,rR)$ directly gives \eqref{eq:doubling}.

In view of Proposition \ref{prop-limit-N}, for any
$\sigma >0$ there exists $r_\sigma \in (0,r_0]$ such that
\begin{equation*}
  \frac{H'(r)}{H(r)}=\frac{2}{r}\mathcal{N}(r)+O(1)
  \le \frac{2 \gamma +\sigma}{r} \quad \text{for all } r \in (0,r_\sigma].
\end{equation*}
Integrating over $(r,r_\sigma)$ and recalling that $H$ is
  continuous in $(0,r_0]$, we deduce \eqref{ineq-H-lower-estimate}.
\end{proof}

\begin{proposition}\label{prop-limit-H}
There exists the limit 
$\lim_{r \to 0^+} r^{-2\gamma}H(r)$ and it is finite.
\end{proposition}
\begin{proof}
  By \eqref{ineq-H-upper-estimate} it is sufficient to show that the
  limit does exist.
  In view of \eqref{D-as-H'} we have that
  \begin{align*}
    &\left(\frac{H(r)}{r^{2\gamma}}\right)'=\frac{r^{2\gamma}H'(r)-2
      \gamma r^{2 \gamma -1}H(r)}{r^{4\gamma}}=2r^{-2\gamma-1}
      (D(r)-\gamma H(r))+r^{-2\gamma}O(1)H(r)\\
    &=2r^{-2\gamma-1}H(r)\left(\mathcal{N}(r)-\gamma+rO(1)\right) \\
     &=2r^{-2\gamma-1}H(r)\left(\int_{0}^r\left[\mathcal{N}'(t)
             -\mathcal{V}(t)\right] \, dt +\int_{0}^r \mathcal{V} (t)
             \, dt+rO(1)\right) \notag
\end{align*}
as $r\rightarrow 0^+$.  Integrating over $(r,\tilde r)$ with $\tilde
  r\in(0,r_0)$ small, we obtain that
  \begin{align}\label{eq:extlim}
    \frac{H(\tilde r)}{\tilde r^{2\gamma}}-\frac{H(r)}{r^{2\gamma}}
      =&\int_r^{\tilde r}2\rho^{-2\gamma-1}H(\rho)
      \left(\int_{0}^\rho \mathcal{V} (t) \, dt\right) \, d\rho \\
    &\notag +\int_r^{\tilde r}\left[
      2\rho^{-2\gamma}H(\rho)O(1)+2\rho^{-2\gamma-1}
      H(\rho)\left(\int_{0}^\rho\left[\mathcal{N}'(t)
      -\mathcal{V}(t)\right] \, dt\right)\right] \, d\rho. 
\end{align}
Letting
\begin{equation*}
  f(\rho) :=2\rho^{-2\gamma}H(\rho)O(1)+2\rho^{-2\gamma-1}H(\rho)
  \left(\int_{0}^\rho\left[\mathcal{N}'(t) -\mathcal{V}(t)\right] \, dt\right), 
\end{equation*}
from \eqref{eq-N'-estimate}, \eqref{ineq-N-bounded} and
\eqref{ineq-H-upper-estimate} it follows that $f \in L^1(0,\tilde r)$ and
hence there exists the limit
\begin{equation*}
  \lim_{r \to 0^+} \int_r^{\tilde r} f(\rho) \, d\rho
  =\int_0^{\tilde r} f(\rho) \, d\rho<+\infty.
\end{equation*}
On the other hand,  in view of \eqref{ineq-V-positive}, there exists the limit 
\begin{equation*}
	\lim_{r \to 0^+} \int_r^{\tilde r}2\rho^{-2\gamma-1}H(\rho)
      \left(\int_{0}^\rho \mathcal{V} (t) \, dt\right) \, d\rho.
\end{equation*}
Therefore we can conclude thanks to \eqref{eq:extlim}. 
\end{proof}

\section{The blow-up analysis} \label{sec-blow-up-analysis}
In the present section, we aim to classify the possible vanishing
orders of solutions to \eqref{prob-extension-straith}. To this
purpose, let $W$ be a non trivial weak solution to
\eqref{prob-extension-straith} and $H$ be defined in \eqref{H}. For
any $ \lambda \in (0,r_0]$, we consider the function
\begin{equation}\label{blown-up-solution}
	V^\la(z):=\frac{W(\la z)}{\sqrt{H(\la)}}.
\end{equation}
It is easy to verify that $V^\la$  weakly solves
\begin{equation*}
\begin{cases}
  \dive(t^{1-2s}\widetilde A(\la \cdot)\nabla V^\la)=0,
  &\text{ on } B_{r_0\la^{-1}}^+, \\
  -\lim_{t \to 0^+} t^{1-2s}\widetilde\alpha(\la \cdot) \pd{V^\la}{t}
  = \kappa_{s,N}\la^{2s}\widetilde{h}(\la\cdot) \Tr(V^\la),
  & \text{ on } B_{r_0\la^{-1}}',
\end{cases}
\end{equation*}
where we have defined $\widetilde\alpha$ in \eqref{alpha}. It follows
that, for any $ \la \in (0,r_0]$,
\begin{equation}\label{eq-blow-up-solution}
  \int_{B^+_1}t^{1-2s}\widetilde A(\la \cdot)\nabla V^\la \cdot \nabla \phi \, dz
  - \kappa_{s,N} \la^{2s} \int_{B'_1} \widetilde{h}(\la \cdot)
  \Tr(V^\la)\Tr(\phi) \, dy=0
\end{equation}
for every $\phi \in H^1_{0,S^+_1}(B_1^+,t^{1-2s})$.
Furthermore by \eqref{H} and \eqref{blown-up-solution}
\begin{equation}\label{eq-V-1-boundary}
  \int_{\mathbb{S}^+} \theta_{N+1}^{1-2s}\mu(\la
  \theta)|V^\la(\theta)|^2 \, dS
  =1 \quad \text{for any  } \la \in (0,r_0].
\end{equation}
\begin{proposition}\label{prop-V-H1-bounded}
For every $R\geq1$,  the family of functions $\{V^\la:\la \in
(0,\frac{\bar r}R]\}$ is bounded in
  $H^1(B_{R}^+,t^{1-2s})$.
\end{proposition}
\begin{proof}
By \eqref{ineq-nabla-E} and \eqref{eq:doubling} we have that,
for all $\lambda\in (0,\frac{\bar r}R]$ with $\bar r$ as in Lemma \ref{l:doub},
\begin{align*}
  \int_{B_R^+} t^{1-2s} |\nabla V^\la|^2 \,
  dz&=\frac{\la^{2s-N}}{H(\la)}\int_{B_{\la R}^+} t^{1-2s}|\nabla W|^2 \,
      dz\leq \frac{\la^{2s-N}R^{\bar c}}{H(\la R)}\int_{B_{\la R}^+}
      t^{1-2s}
      |\nabla W|^2 \,
      dz\\
    &\le \frac{2 R^{\bar c+N-2s}}{1-2 \kappa_{N,s}\eta_{\tilde{h}}(\la R)}
      \mathcal{N}(\la R)
      +\frac{2(N-2s) R^{\bar c+N-2s}\kappa_{N,s}
      \eta_{\tilde{h}}(\la R)}{1-2\kappa_{N,s}
      \eta_{\tilde{h}}(\la R)},
\end{align*}
which, together with \eqref{etapiccolo} and \eqref{ineq-N-bounded},
allows us to deduce that $\{\nabla V^\la:\la \in (0,\frac{\bar r}R]\}$ is
uniformly bounded in $L^2(B_R^+,t^{1-2s})$. On the other hand,
  \eqref{mu-estimates}, a scaling argument, and \eqref{eq:doubling} imply that
  \begin{equation*}
    \int_{S_R^+}t^{1-2s}|V^\lambda|^2dS=\frac{\lambda^{-N-1+2s}}{H(\lambda)}
    \int_{S_{R\lambda}^+}t^{1-2s}W^2dS\leq2 R^{N+1-2s}
    \frac{H(R\lambda)}{H(\lambda)}\leq2 R^{N+1-2s+\bar c},
  \end{equation*}
  so that the claim follows from \eqref{eq-Poincare}.
\end{proof}

\begin{proposition}\label{prop-blow-up-sub}
  Let $W$ be a non trivial weak solution to
  \eqref{prob-extension-straith}. Let $\gamma$ be as in Proposition
  \ref{prop-limit-N}. There exists
  $m_0 \in \mathbb{N} \setminus \{0\}$ (which is odd
  in the case $N=1$) such that
\begin{equation}\label{eq-gamma-integer}
\gamma=m_0.
\end{equation}
Furthermore, for any sequence $\{\lambda_n\}$ such that $\la_n \to
0^+$ as $n\to\infty$, there
exist a subsequence $\{\la_{n_k}\}$ and an eigenfunction $\Psi$  of problem
\eqref{prob-eigenvalues} associated
with the eigenvalue $\mu_{m_0}=m_0^2+m_0 (N-2s)$
such that 
$\norm{\Psi}_{L^2(\mathbb{S}^+,\theta_{N+1}^{1-2s})}=1$ and
\begin{equation}\label{limit-blow-up-H}
\frac{W(\la_{n_k} z)}{\sqrt{H(\la_{n_k})}} \to |z|^{\gamma}\Psi\left(\frac{z}{|z|}\right)  \text{ as } k \to +\infty \quad \text{ strongly in } H^1(B_1^+,t^{1-2s}).
\end{equation}
\end{proposition}

\begin{proof}
  Let $W$ be a non trivial weak solution to
  \eqref{prob-extension-straith} and $\{\lambda_n\}$ be a sequence
  such that $\la_n \to
0^+$ as $n\to+\infty$.
Thanks to Proposition \ref{prop-V-H1-bounded}, there exist a
subsequence $\{\la_{n_k}\}$ and
$V \in H^1(B_1^+,t^{1-2s})$ such that
\begin{equation}\label{eq-blow-up-main:1}
V^{\la_{n_k}} \rightharpoonup V \quad \text{ weakly in } H^1(B_1^+,t^{1-2s})\text{ as } k \to +\infty.	
\end{equation}
Observing that $\la_{n_k} \in (0,r_0)$ and
thus $B^+_1\subset B^+_{r_0/\lambda_{n_k}}$ for sufficiently large
$k$, from \eqref{eq-blow-up-solution} we deduce that, for sufficiently
large $k$,
\begin{equation}\label{eq-blow-up-main:1.1}
  \int_{B_1^+}t^{1-2s}\widetilde A(\la_{n_k}\cdot) \nabla V^{\la_{n_k}}
  \cdot \nabla \phi \, dz =\kappa_{s,N}\la_{n_k}^{2s}
  \int_{B_1'}\widetilde{h}(\la_{n_k}\cdot)\Tr( V^{\la_{n_k}}) \Tr(\phi)\, dy
\end{equation}
for every $\phi \in H^1_{0,S_1^+}(B_1^+,t^{1-2s})$. In order to study
what happens as $k\to +\infty$, we notice that the term on the left
hand side of \eqref{eq-blow-up-main:1.1} can be rewritten as follows
\begin{align}\label{eq-blow-up-main:2}
  &\int_{B_1^+}t^{1-2s}\widetilde A(\la_{n_k}\cdot) \nabla
    V^{\la_{n_k}}
    \cdot \nabla \phi \, dz \\
  &=\int_{B_1^+}t^{1-2s}(\widetilde A(\la_{n_k}\cdot)
    -\mathrm{Id} _{N+1} )\nabla V^{\la_{n_k}}\cdot \nabla \phi \, dz+
    \int_{B_1^+}t^{1-2s}\nabla V^{\la_{n_k}}\cdot \nabla \phi \, dz. \notag
\end{align}
Therefore, in view of \eqref{A-O}, Proposition \ref{prop-V-H1-bounded}
and \eqref{eq-blow-up-main:1}, we conclude that
\begin{equation}\label{eq-blow-up-main:3}
  \lim_{k \to +\infty}\int_{B_1^+}t^{1-2s}\widetilde A(\la_{n_k}\cdot)
  \nabla V^{\la_{n_k}}\cdot \nabla \phi \, dz=\int_{B_1^+} t^{1-2s}
  \nabla V\cdot \nabla \phi \, dz.
\end{equation}
As for the right hand side in \eqref{eq-blow-up-main:1.1}, we have that
\begin{align}\label{eq-blow-up-main:4}
  &\bigg|\la_{n_k}^{2s}\int_{B_1'}\widetilde{h}(\la_{n_k}\cdot)
    \Tr( V^{\la_{n_k}}) \Tr(\phi)\, dy\bigg| \\
  \notag\le \la_{n_k}^{2s}&\eta_{\tilde h(\la_{n_k}\cdot)}(1)
    \left(\int_{B_1^+}t^{1-2s} |\nabla \phi|^2\, dy\right)^{\!\!\frac{1}{2}}
    \!\!\left(\int_{B_1^+}t^{1-2s} |\nabla V^{\la_{n_k}} |^2\, dz
    +\frac{N-2s}{2}\int_{\mathbb{S}^+}
    \theta_{N+1}^{1-2s}|V^{\la_{n_k}}|^2
    \, dS\right)^{\!\!\frac{1}{2}}
\end{align}
thanks to H\"older's inequality and \eqref{ineq-found}. By
\eqref{eta-f} and the change of variable
$x\mapsto \lambda_{n_k}x$, we obtain that
\begin{align}\label{eq-blow-up-main:5}
  \la_{n_k}^{2s}\eta_{\tilde h(\la_{n_k}\cdot)}(1) &=
    \mathcal S_{N,s}\omega_N^{\frac{4
    s^2\e}{N(N+2s\e)}}\lambda_{n_k}^{2s}
    \Vert\widetilde h(\la_{n_k}\cdot)\Vert_{L^{\frac{N}{2s}+\e}(B_1')}\\
&= \mathcal S_{N,s}\omega_N^{\frac{4 s^2\e}{N(N+2s\e)}}\Vert\widetilde h\Vert_{L^{\frac{N}{2s}+\e}(B_{\la_{n_k}}')} \la_{n_k}^{\frac{4s^2 \e}{N+2s\e}}.\notag
\end{align}
Putting together \eqref{eq-blow-up-main:4} and
\eqref{eq-blow-up-main:5}, thanks to Proposition
\ref{prop-V-H1-bounded}, \eqref{eq-V-1-boundary}, and
\eqref{mu-estimates} we infer that
\begin{equation}\label{eq-blow-up-main-5.1}
  \lim_{k \to +\infty}	\la_{n_k}^{2s}\int_{B_1'}
  \widetilde{h}(\la_{n_k}\cdot)\Tr( V^{\la_{n_k}}) \Tr(\phi)\, dy=0.
\end{equation}
Passing to the limit as $k \to +\infty$ in \eqref{eq-blow-up-main:1.1}
we conclude that $V$ weakly solves the following problem:
\begin{equation}\label{eq-blow-up-main:6}
\begin{cases}
\dive(t^{1-2s}\nabla V)=0, &\text{in } B_1^+,\\
\lim_{t \to 0^+} t^{1-2s}\pd{V}{t}=0, &\text{on } B_1'.
\end{cases}
\end{equation}
In particular $V$ is smooth on $B_1^+$ and $V\not \equiv 0$ since, by
\eqref{mu-O}, \eqref{eq-blow-up-main:1} and the compactness of the
trace operator  in \eqref{trace-Sr+}, \eqref{eq-V-1-boundary}
leads to
\begin{equation}\label{eq-V-1-boundary-limit-profile}
	\int_{\mathbb{S}^+} \theta_{N+1}^{1-2s}V^2 \, dS =1.
\end{equation}
Now we aim to show that, along a further subsequence,
\begin{equation}\label{eq-blow-up-main:7}
	V^{\la_{n_k}} \to V \quad \text{strongly in } H^1(B_1^+,t^{1-2s})\text{ as } k \to +\infty.	
\end{equation}
To this purpose, we first notice that a change of variables in
\eqref{formula-div} yields
\begin{multline}\label{eq-blow-up-main:8}
  \int_{B_1^+ }t^{1-2s} \widetilde A(\la_{n_k}\cdot) \nabla
    V^{\la_{n_k}} \cdot \nabla \phi \, dz-\int_{\mathbb{S}^+} \theta_{N+1}^{1-2s}
    \widetilde A(\la_{n_k }\cdot) \nabla V^{\la_{n_k}} \cdot z \, \phi \, dS\\
    =\kappa_{s,N}\la_{n_k}^{2s}\int_{B_1'} \widetilde h(\la_{n_k}\cdot)
    \Tr(V^{\la_{n_k}}) \Tr(\phi) \, dy 
\end{multline}
for any $\phi \in H^1(B_1^+,t^{1-2s})$ and $k$ sufficiently
  large.

From Proposition \ref{prop-V-H1-bounded} and the regularity result
contained in \cite[Theorem 2.1]{fellisiclari} and recalled in Remark
\ref{rem:regularity}, it follows that, for $k$ sufficiently large,
$\{\nabla _x V^{\lambda_{n_k}}\}$ and
$\big\{t^{1-2s}\frac{\partial V^{\lambda_{n_k}}}{\partial t}\big\}$
are bounded uniformly with respect to $k$ in the spaces
$H^1(B^+_1, t^{1-2s})$ and $H^1(B^+_1,t^{2s-1})$ respectively.  Then,
by the continuity of the trace operator $\mathop{\rm Tr}_1$ from
$H^1(B^+_1, t^{1-2s})$ to $L^2(\mathbb{S}^+,\theta_{N+1}^{1-2s})$
and from $H^1(B^+_1, t^{2s-1})$ to
  $L^2(\mathbb{S}^+,\theta_{N+1}^{2s-1})$, we have that
  $\{\mathop{\rm Tr}_1(\nabla _x V^{\lambda_{n_k}})\}$ is bounded in
  $\big(L^2(\mathbb{S}^+, \theta_{N+1}^{1-2s})\big)^N$ and
  $\big\{t^{1-2s}\frac{\partial V^{\lambda_{n_k}}}{\partial t}\big\}$
  is bounded in $L^2(\mathbb{S}^+, \theta_{N+1}^{2s-1})$. Therefore
\begin{equation*}
  \int_{\mathbb{S}^+}\theta_{N+1}^{1-2s}|\nabla V^{\lambda_{n_k}}|^2\,
  dS=\int_{\mathbb{S}^+}\theta_{N+1}^{1-2s}|\nabla_x
  V^{\lambda_{n_k}}|^2\,dS+\int_{\mathbb{S}^+}\theta_{N+1}^{2s-1}\left|
    \theta_{N+1}^{1-2s}\frac{\partial V^{\lambda_{n_k}} }{\partial
        t}\right|^2\, dS
\end{equation*}
is bounded uniformly with respect to $k$.  Taking into account
\eqref{A-O}, it follows that there exists
$f \in L^2(\mathbb{S}^+,\theta_{N+1}^{1-2s})$ such that, up to a
further subsequence,
\begin{equation}\label{eq-blow-up-main:9}
  \widetilde A(\la_{n_k }\cdot) \nabla V^{\la_{n_k}} \cdot z \rightharpoonup
  f \quad \text{ weakly in } L^2(\mathbb{S}^+,\theta_{N+1}^{1-2s})
  \text{ as } k \to +\infty.
\end{equation}
Thus by \eqref{eq-blow-up-main:3} and after proving
\eqref{eq-blow-up-main-5.1} when $\phi\in H^1(B^+_1, t^{1-2s})$ with
the same argument (i.e. combining \eqref{ineq-found} with
    \eqref{eq-blow-up-main:5}), passing to the limit as
$ k \to +\infty$ in \eqref{eq-blow-up-main:8} we obtain that
\begin{equation}\label{eq-blow-up-main:10}
  \int_{B_1^+ }t^{1-2s}  \nabla V\cdot \nabla \phi \, dz
  =\int_{\mathbb{S}^+} \theta_{N+1}^{1-2s} f \phi \, dS
\end{equation}
for any $\phi \in H^1(B_1^+,t^{1-2s})$.  Furthermore, by
\eqref{eq-blow-up-main:9}, combined with \eqref{eq-blow-up-main:1} and
compactness of the trace operator in \eqref{trace-Sr+}, we have that
\begin{equation}\label{eq-blow-up-main:11}
\lim_{k \to +\infty}\int_{\mathbb{S}^+} t^{1-2s} \widetilde A(\la_{n_k
}\cdot)
\nabla V^{\la_{n_k}} \cdot z \  V^{\la_{n_k}}\, dS=\int_{\mathbb{S}^+}
t^{1-2s} f V \, dS.
\end{equation}
Hence, testing \eqref{eq-blow-up-main:8} with $V^{\la_{n_k}}$ itself,
taking into account \eqref{eq-blow-up-main:11}, using
\eqref{eq-blow-up-main-5.1} with $\phi=V^{\lambda_{n_k}}$, and passing
to the limit as $ k \to +\infty$, we deduce that
\begin{equation*}
  \lim_{k \to +\infty}\int_{B_1^+ }t^{1-2s} \widetilde A(\la_{n_k}\cdot)
  \nabla V^{\la_{n_k}} \cdot \nabla V^{\la_{n_k}}\,
  dz=\int_{\mathbb{S}^+}
  t^{1-2s} f V \, dS,
\end{equation*}
which, by \eqref{eq-blow-up-main:10} tested with $V$, implies  that 
\begin{equation}\label{eq-blow-up-main:13}
\lim_{k \to +\infty}\int_{B_1^+ }t^{1-2s} A(\la_{n_k}\cdot) \nabla
V^{\la_{n_k}}
\cdot \nabla V^{\la_{n_k}}\, dz=\int_{B_1^+ }t^{1-2s} |\nabla V|^2 dz.
\end{equation}
Writing the left hand side in \eqref{eq-blow-up-main:13} as in
\eqref{eq-blow-up-main:2}, by \eqref{A-O} and
Proposition \ref{prop-V-H1-bounded} we infer that
\begin{equation*}
  \lim_{k \to +\infty}\int_{B_1^+ }t^{1-2s} | \nabla V^{\la_{n_k}} |^2\,
  dz
  =\int_{B_1^+ }t^{1-2s} |\nabla V|^2 dz.
\end{equation*}
This convergence, together with \eqref{eq-blow-up-main:1}, allows us
to conclude that $\nabla V^{\lambda_{n_k}}\to \nabla V $ in
$L^2(B^+_1,t^{1-2s})$.  In conclusion, combining this with the
compactness of the trace operator given in \eqref{trace-Sr+},
\eqref{eq-blow-up-main:7} easily follows from Remark
\ref{remarknormaequiv}.

For any $ r \in (0,1]$ and $k\in\mathbb{N}$ we define
\begin{align*}
  &H_k(r):=\frac1{r^{N+1-2s}}\int_{S_r^+}t^{1-2s}
    \mu(\la_{n_k}\cdot) |V^{\la_{n_k}}|^2 \, dS, \\
  &D_k(r):=\frac1{r^{N-2s}}\!\left(\!\int_{B_r^+}\!t^{1-2s}\widetilde
    A(\la_{n_k}
    \cdot)
    \nabla V^{\la_{n_k}}\!\cdot\! \nabla V^{\la_{n_k}} dz
    -k_{s,N}\la_{n_k}^{2s}
    \int_{B_r'}\!\widetilde{h}(\la_{n_k}\cdot)
    |\Tr(V^{\la_{n_k}})|^2 \, dy\right),
\end{align*}
and
\begin{equation*}
  H_V(r):=\frac1{r^{N+1-2s}}\int_{S_r^+}t^{1-2s}V^2 \, dS, \quad
  D_V(r):=\frac1{r^{N-2s}}\int_{B_r^+}t^{1-2s}|\nabla V|^2\, dz.
\end{equation*}
By Proposition \ref{H>0} in the case $\widetilde{h}=0$,
$\widetilde A=\mathrm{Id}_{N+1}$ and $\mu=1$, it is clear that
$H_V(r)>0$ for any $r \in (0,1]$. Thus the frequency function
\begin{equation*}
  \mathcal{N}_V(r):=\frac{D_V(r)}{H_V(r)}\quad
  r \in (0,1]
\end{equation*}
is well defined.  Furthermore by \eqref{limit-N},
\eqref{eq-blow-up-main:7}, a change of variables, and a
  combination of \eqref{ineq-found} and
    \eqref{eq-blow-up-main:5}, we have that
\begin{equation}\label{eq-blow-up-main:14}
\gamma=\lim_{k \to +\infty} \mathcal{N}(\la_{n_k}r)=\lim_{k \to
  +\infty}
\frac{D_k(r)}{H_k(r)}=\mathcal{N}_V(r) \quad \text{ for any } r \in (0,1]
\end{equation}
and hence $\mathcal{N}_V'(r)=0$ for a.e. $ r \in (0,1]$. 
Arguing as in  Proposition \ref{prop-N'} in the case $\widetilde{h}=0$, 
$\widetilde A=\mathrm{Id}_{N+1}$ and $\mu=1$, we
can prove that 
\begin{equation*}
  \mathcal{N}_V'(r)=2r\frac{ \left(\int_{S_r^+}t^{1-2s}V^2\,
      dS\right)
    \left(\int_{S_r^+}t^{1-2s}|\nabla V \cdot \nu|^2\, dS
    \right)-\left(\int_{S_r^+}t^{1-2s}V (\nabla V \cdot \nu)
      \, dS\right)^2}{\left(\int_{S_r^+}t^{1-2s}V^2\, dS\right)^2}.
\end{equation*}
Therefore we conclude that
\begin{equation*}
  \left(\int_{S_r^+}t^{1-2s}V^2\, dS\right)\left(\int_{S_r^+}t^{1-2s}
    |\nabla V\cdot \nu|^2 \, dS \right)=\left(\int_{S_r^+}t^{1-2s}V\,
    (\nabla V\cdot \nu)  \, dS\right)^2 \quad \text{ a.e. } r \in (0,1)
\end{equation*}
where $\nu=\frac{z}{|z|}$, i.e. equality holds in the Cauchy-Schwartz
inequality for the vectors $V$ and $\nabla V\cdot \nu$ in
$L^2(S_r^+,t^{1-2s})$ for a.e. $r \in (0,1)$.  It follows that there
exists a function $\rho(r)$ defined a.e. such that, writing $V$ in
polar coordinates,
\begin{equation}\label{eq-blow-up-main:16}
  \pd{V}{r}(r\theta) =\rho(r)V(r\theta)\quad \text{for a.e. } r\in(0,1)
  \text{ and for any }  \theta\in\mathbb{S}^+.
\end{equation} 
By \eqref{eq-blow-up-main:16} we have that
\begin{equation}\label{eq-blow-up-main:17}
  \int_{S^+_r}t^{1-2s}V (\nabla V \cdot \nu )\, dS =
  \rho(r) \int_{S^+_r}t^{1-2s} V^2\, dS.
\end{equation}
In the case $\widetilde{h}=0$, $A=\mathrm{Id}_{N+1}$ and
$\mu=1$, \eqref{H'1} boils down to
$H'_V = \frac{2}{r^{N+1-2s}}\int_{S^+_r} t^{1-2s} V\frac{\partial
  V}{\partial\nu}\,dS $, since the perturbative term involves
$\nabla \mu$, which now trivially equals 0. From this and
\eqref{eq-blow-up-main:17} we deduce that
$\rho(r)= \frac{H'_V(r)}{2H_V(r)}$.  At this point, we exploit
\eqref{D-as-H'} which, in the case 
$\widetilde{h}=0$, $A=\mathrm{Id}_{N+1}$ and $\mu=1$, becomes 
$H'_V(r)=\frac{2}{r}D_V(r)$ and thus implies 
\begin{equation*}
\rho(r)=\frac{1}{r}\mathcal{N}_V(r)= \frac{\gamma}{r},
\end{equation*}
where we used also \eqref{eq-blow-up-main:14}.  Then an integration
over $(r,1)$ of \eqref{eq-blow-up-main:16} for any fixed
$\theta \in \mathbb{S}^+$ yields
\begin{equation}\label{eq-blow-up-main:19}
  V(r\theta)=r^{\gamma}V(\theta)=
  r^{\gamma}\Psi(\theta) \quad
  \text{ for any } (r,\theta)\in(0,1] \times \mathbb{S}^+,
\end{equation}
where $\Psi:=\restr{V}{\mathbb{S}^+}$.  In view of \cite[Lemma
2.1]{Fall-Felli-2014}, 
\eqref{eq-blow-up-main:6} becomes
\begin{equation*}
  \gamma(N-2s+\gamma)r^{-1-2s+\gamma}\theta_{N+1}^{1-2s}
  \Psi(\theta)+r^{-1-2s+\gamma}\mathop{\rm{div}_{\mathbb{S}^+}}
  (\theta_{N+1}^{1-2s}\nabla_{\mathbb{S^+}}\Psi(\theta))=0 
\end{equation*}
for any $(r,\theta)\in(0,1] \times \mathbb{S}^+$, together with
  the boundary condition
  $\lim_{\theta_{N+1} \to 0^+}\theta_{N+1}^{1-2s}\,\nabla_{\mathbb{S}}
  \Psi\cdot\nu=0$ on $\mathbb{S}'$.  Since $V^\la$ is odd with
respect to $y_N$ for any $\la \in(0,r_0]$ by \eqref{blown-up-solution}
and \eqref{W}, then also $V$ is odd with respect to $y_N$, so
  that $\Psi \in H_{\rm odd}^1(\mathbb{S}^+,\theta_{N+1}^{1-2s})$.
By \eqref{eq-blow-up-main:19} and
\eqref{eq-V-1-boundary-limit-profile} we have that
$\norm{\Psi}_{L^2(\mathbb{S}^+,\theta_{N+1}^{1-2s})}= 1$, so that
$\Psi\not\equiv 0$ is an eigenfunction of problem
\eqref{prob-eigenvalues} associated to the eigenvalue
$\gamma(\gamma+N-2s)$.  From \eqref{eigenvalues} it follows that there
exists $m_0 \in \mathbb{N}\setminus \{0\}$ (which is odd
  in the case $N=1$) such that
$\gamma(\gamma+N-2s)=m_0(m_0+N-2s)$.  Therefore, since
  $\gamma\geq0$ by Proposition \ref{prop-limit-N}, we conclude that
  $\gamma=m_0$ thus proving \eqref{eq-gamma-integer}. Moreover
\eqref{limit-blow-up-H} follows from \eqref{eq-blow-up-main:7}
  and \eqref{eq-blow-up-main:19}.
\end{proof}

In Proposition \ref{prop-limit-H} we have shown that there exists the
limit $\lim_{\la \to 0^+}\la^{-2\gamma}H(\la)$ and it is non-negative.
Now we prove that $\lim_{\la \to 0^+}\la^{-2\gamma}H(\la)>0$.

To this end we define, for every $ \la \in(0,r_0]$,
$m\in\mathbb{N}\setminus \{0\}$, $k\in \{1,\dots, M_m\}$,
\begin{equation}\label{Fourier-coefficents}
  \varphi_{m,k}(\la):=\int_{\mathbb{S}^+}\theta_{N+1}^{1-2s}
  W(\la \theta) Y_{m,k}(\theta) \, dS,
\end{equation}
i.e. $\{\varphi_{m,k}(\la)\}_{m,k }$ are the Fourier coefficients of
$W(\la \cdot)$ with respect to the orthonormal basis
$\{Y_{m,k}\}_{m,k }$ introduced in \eqref{orthonormal-base}.
For every $\la \in (0,r_0]$, $m\in\mathbb{N}\setminus \{0\}$,
$k\in \{1,\dots, M_m\}$, we also define 
\begin{align}\label{Upsilon}
  \Upsilon_{m,k}(\la):=&-\int_{B_{\la}^+} t^{1-2s} (\widetilde A-\mathop{\rm Id}
    \nolimits_{N+1})\nabla W \cdot \frac{1}{|z|}\nabla_{\mathbb{S}}
    Y_{m,k}\big(\tfrac{z}{|z|}\big)\, dz  \\
   &+\int_{S_{\la}^+} t^{1-2s} (\widetilde A-\mathop{\rm Id}
                         \nolimits_{N+1})\nabla W
    \cdot \frac{z}{|z|}Y_{m,k}\big(\tfrac{z}{|z|}\big)\, dS\notag\\
  &+\kappa_{N,s} \int_{B_{\la}'}\widetilde{h}(y) \Tr(W)
    \Tr\left( Y_{m,k}\big(\tfrac{y}{|y|}\big)\right)\, dy\notag, 
\end{align}
where $\mathop{\rm Id}\nolimits_{N+1}$ is the identity
$(N+1)\times (N+1)$ matrix.
\begin{proposition}
 Let $\gamma$ be as in \eqref{limit-N} and let
  $m_0\in \mathbb{N} \setminus \{0\}$ be such that $\gamma=m_0$
  according to Proposition \ref{prop-blow-up-sub}. For every
  $k \in \{1, \dots, M_{m_0}\}$ and $r \in (0,r_0]$
\begin{multline}\label{eq-Fourier-coefficents}
  \varphi_{m_0,k}(\la)=\la^{m_0}\left(\frac{\varphi_{m_0,k}(r)}{r^{m_0}}
    +\frac{m_0r^{-2m_0-N+2s}}{2m_0+N-2s}\int_0^r\rho^{m_0-1}
    \Upsilon_{m_0,k}(\rho) \, d\rho\right)\\
  +\la^{m_0}\frac{m_0+N-2s}{2m_0+N-2s}\int_{\la}^r\rho^{-m_0-N-1+2s}
  \Upsilon_{m_0,k}(\rho)
  \, d\rho + O\left(\la^{m_0+\frac{4s^2\e}{N+2s\e}}\right)
\end{multline}
as $\la \to 0^+$.
\end{proposition}
\begin{proof}
  Let $k \in \{1, \dots, M_{m_0}\}$ and $\phi \in
  \mathcal{D}(0,r_0)$. Testing \eqref{eq-extension-straith} with
  $|z|^{-N-1+2s} \phi(|z|)Y_{m_0,k}\big(\frac{z}{|z|}\big)$, since
  $Y_{m_0,k}$ solves \eqref{eq-egienvlulues}, we obtain that
  $\varphi_{m_0,k}$ satisfies
\begin{equation}\label{eq-Fourier:1}
  -\varphi''_{m_0,k}-\frac{N+1-2s}{\la}\varphi'_{m_0,k}+
  \frac{\mu_{m_0}}{\la^2}\varphi_{m_0,k}= \zeta_{m_0,k}
\end{equation} 
in the sense of distributions in $(0,r_0)$, where 
\begin{multline*}
  \sideset{_{\mathcal{D}'(0,r_0)}}{_{\mathcal{D}(0,r_0)}}{\mathop{\langle\zeta_{m_0,k},
    \phi\rangle}}:= 
    \kappa_{N,s}\int_0^{r_0}\frac{\phi(\la)}{\la^{2-2s}}
    \left(\int_{\mathbb{S}'}
    \widetilde{h}(\la \theta') \Tr(W(\la\cdot))(\theta')
    Y_{m_0,k}(\theta',0)
    \, dS' \right)d\la\\
  -\int_0^{r_0}\left(\int_{S^+_\lambda}  t^{1-2s} (A-\mathop{\rm
    Id}\nolimits_{N+1})
    \nabla W \cdot\nabla( |z|^{-N-1+2s}
    \phi(|z|)Y_{m_0,k}\big(\tfrac{z}{|z|}
    \big))\,dS \right) d\la. 
\end{multline*}
Furthermore, it is easy to verify that $\Upsilon_{m_0,k} \in L^1(0,r_0)$
and 
\begin{equation*}
\Upsilon_{m_0,k}'(\la)=\la^{N+1-2s}\zeta_{m_0,k}(\la)
\end{equation*}
in the sense of distributions in $(0,r_0)$. Then equation
\eqref{eq-Fourier:1} can be rewritten as follows
\begin{equation}\label{eq-Fourier:4}
  -(\la^{2m_0+N+1-2s}(\la^{-m_0}\varphi_{m_0,k}(\la))')'
  =\la^{m_0}\Upsilon_{m_0,k}'(\la)
\end{equation}
in the sense of distributions in $(0,r_0)$. Integrating
\eqref{eq-Fourier:4} over $(\la,r)$ for any $r \in (0,r_0]$, we obtain
that there exists a constant $ c_{m_0,k}(r)\in \mathbb{R}$ which
depends only on $m_0,k,r$, such that
\begin{multline*}
(\la^{-m_0}\varphi_{m_0,k}(\la))'=-\la^{-m_0-N-1+2s}\Upsilon_{m_0,k}(\la)
\\-m_0\la^{-2m_0-N-1+2s}\left( c_{m_0,k}(r)+
  \int_\la^{r}\rho^{m_0-1}\Upsilon_{m_0,k}(\rho) \, d\rho\right)
\end{multline*}
in the sense of distributions in $(0,r_0)$. In particular we deduce
that $\varphi_{m_0,k}\in W^{1,1}_{\rm loc}((0,r_0])$ and a further
integration over $(\la,r)$ gives 
\begin{align}\label{eq-Fourier:6}
\varphi_{m_0,k}(\la)=&\la^{m_0}\left(\frac{\varphi_{m_0,k}(r)}{r^{m_0}}-\frac{m_0c_{m_0,k}(r)}{(2m_0+N-2s)r^{2m_0+N-2s}}\right)\\
&+\la^{m_0}\frac{m_0+N-2s}{2m_0+N-2s}\int_{\la}^r\rho^{-m_0-N-1+2s}\Upsilon_{m_0,k}(\rho) \, d\rho \notag \\
&+ \frac{m_0\la^{-m_0-N+2s}}{2m_0+N-2s}\left(c_{m_0,k}(r)+\int_\la^r\rho^{m_0-1}\Upsilon_{m_0,k}(\rho) \, d\rho\right) \notag
\end{align}
for every $\la,r \in (0,r_0]$. 
Now we claim that  
\begin{equation}\label{eq-Fourier:7}
\int_0^{r_0}\rho^{-m_0-N-1+2s}|\Upsilon_{m_0,k}(\rho)| \, d\rho <+\infty.
\end{equation}
By the H\"older inequality, a change of variables, \eqref{A-O},
\eqref{blown-up-solution}, Proposition \ref{prop-V-H1-bounded}, and
\eqref{ineq-H-upper-estimate} we have that
\begin{align}\label{eq-Fourier:8}
  &\la^{-m_0-N-1+2s}\left|\int_{B_{\la}^+} t^{1-2s} (\widetilde A-\mathop{\rm
    Id}\nolimits_{N+1})
    \nabla W \cdot \frac{1}{|z|}\nabla_{\mathbb{S}}
    Y_{m_0,k}\big(\tfrac{z}{|z|}\big)\, dz
    \right| \\
  & \le\la^{-m_0-N-1+2s}\left(\int_{B_{\la}^+}\! t^{1-2s}
    |(\widetilde A-\mathop{\rm Id}
    \nolimits_{N+1})\nabla W|^2 \, dz\right)^{\!\!\frac{1}{2}} \!\!\left(\int_{B_{\la}^+}\!
    \frac{t^{1-2s}}{|z|^2}\left|\nabla_{\mathbb{S}}Y_{m_0,k}\big(\tfrac{z}{|z|}\big)\right|^2
    \,dz\right)^{\!\!\frac{1}{2}} \notag\\
  &\le\la^{-m_0-1}O(\la)\sqrt{H(\la)}\left(\int_{B_{1}^+} t^{1-2s}
    |\nabla V^\la|^2 \, dz
    \right)^{\!\!\frac{1}{2}}\!\!
    \left(\int_{B_1^+}\frac{t^{1-2s}}{|z|^2}
    \left|\nabla_{\mathbb{S}}Y_{m_0,k}\big(\tfrac{z}{|z|}\big)\right|^2\,dz
    \right)^{\!\!\frac{1}{2}} \notag\\
  & \le {\rm{const}} \, \la^{-m_0}\sqrt{H(\la)} \le \mathop{\rm{const}}  ,\notag
\end{align}
where we used the fact that
\begin{equation}
\begin{split}
  \int_{B_1^+}\frac{t^{1-2s}}{|z|^2}
  \left|\nabla_{\mathbb{S}}Y_{m_0,k}
    \big(\tfrac{z}{|z|}\big)\right|^2\,dz &= \int_0^1
  \rho^{N-1-2s}\left(
    \int_{\mathbb{S}^+}
    \theta_{N+1}^{1-2s}|\nabla_{\mathbb{S}}Y_{m_0,k}(\theta)|^2\, dS
  \right)\, d\rho\\
  &= \frac{m_0^2+m_0(N-2s)}{N-2s}.
\end{split}
\end{equation}
Dealing with the second term of \eqref{Upsilon}, 
from an integration by parts, the H\"older inequality,
\eqref{A-O} \eqref{blown-up-solution}, Proposition
\ref{prop-V-H1-bounded}, and \eqref{ineq-H-upper-estimate} it follows
that, for every $r\in (0,r_0]$,
\begin{align}\label{eq-Fourier:9}
  \int_0^{r}&\la^{-m_0-N-1+2s}\left|\int_{S_{\la}^+} t^{1-2s}
    (\widetilde A-
    \mathop{\rm Id}\nolimits_{N+1})\nabla W \cdot
    \frac{z}{|z|}Y_{m_0,k} \big(\tfrac{z}{|z|}\big)\, dS
    \right|d\la\\
  \le&\mathop{\rm const}\int_0^{r}\la^{-m_0-N+2s}\left(\int_{S_{\la}^+}t^{1-2s}
       |\nabla W|\left|Y_{m_0,k} \big(\tfrac{z}{|z|}\big)\right|\, dS\right)d\la \notag\\
  =&\mathop{\rm const}\bigg(r^{-m_0-N+2s}\int_{B_{r}^+}t^{1-2s}|\nabla
     W|\left|Y_{m_0,k} \big(\tfrac{z}{|z|}\big)\right|\, dz
     \notag\\
  &\qquad\qquad+(m_0+N-2s)\int_0^{r}\la^{-m_0-N-1+2s}\bigg(\int_{B_{\la}^+}t^{1-2s}
    |\nabla W|\left|Y_{m_0,k}\big(\tfrac{z}{|z|}\big)\right|\, dz\bigg)d\la\bigg)\notag\\ 
  \leq&\mathop{\rm const}\left(r^{-m_0+1}\sqrt{H(r)}+\int_0^{r}\la
     ^{-m_0}\sqrt{H}(\la)
     \,d\la\right)\le \mathop{\rm const}\,r,\notag  
\end{align}
taking into account that 
\begin{equation}
  \int_{B^+_\la}t^{1-2s}\left|Y_{m_0,k}\big(\tfrac{z}{|z|}\big)\right|^2\, dz
  =\frac{\lambda^{N+2-2s}}{N+2-2s}. 
\end{equation}
By the
H\"older inequality the third term in \eqref{Upsilon} can be estimated as
\begin{align}\label{eq-Fourier:10}
  &\la^{-m_0-N-1+2s}\left|\int_{B_{\la}'}\widetilde{h}(y)
    \Tr(W)\Tr\left( Y_{m_0,k}
    \big(\tfrac{y}{|y|}\big)\right)\, dy\right| \\
  &\le \la^{-m_0-N-1+2s}\left(\int_{B_{\la}'}|\widetilde{h}(y)| |\Tr(W)|^2\,dy
    \right)^{\!\!\frac{1}{2}}\left(\int_{B_{\la}'}|\tilde{h}(y)|\left|\Tr\left(
    Y_{m_0,k}\big(\tfrac{y}{|y|}\big)
    \right)\right|^2 dy\right)^{\!\!\frac{1}{2}} \notag \\
  &\le \la^{-m_0-N-1+2s}
    \eta_{|\tilde{h}|}(\la)\left(\int_{B_\la^+}t^{1-2s} |\nabla W|^2\,
    dz
    +\frac{N-2s}{2\la}\int_{S_\la^+}t^{1-2s} W^2 \, dS\right)^{\!\!\frac{1}{2}}\times\notag \\
  & \qquad\qquad\times \left(\int_{B_\la^+}t^{1-2s} \left|\nabla
    Y_{m_0,k}
    \big(\tfrac{z}{|z|}\big)\right|^2\,
    dz+\frac{N-2s}{2\la}
    \int_{S_\la^+}
    t^{1-2s}\left|Y_{m_0,k}\big(\tfrac{z}{|z|}\big)\right|^2
    \, dS\right)^{\!\!\frac{1}{2}}\notag \\
  &\le
    \la^{-m_0-1}\eta_{|\tilde{h}|}(\la)\sqrt{H(\la)}\left(\int_{B_1^+}\!t^{1-2s}
    |\nabla V^\la|^2 dz+(N-2s)\!\int_{\mathbb{S}^+}\!\theta_{N+1}^{1-2s}
    \mu(\la \theta)|V^\la|^2 \, dS\right)^{\!\!\frac{1}{2}}\!\times\notag \\
  & \qquad\qquad\times \left(\la^2\int_{B_1^+}t^{1-2s} \left|\nabla
    Y_{m_0,k}\big(\tfrac{z}{|z|}\big)\right|^2\, dz+
    \frac{N-2s}{2}\int_{\mathbb{S}^+}
    \theta_{N+1}^{1-2s}|Y_{m_0,k}(\theta)|^2 \, dS
    \right)^{\!\!\frac{1}{2}}\notag \\
  &\le \mathop{\rm const}
    \la^{-m_0-1}\eta_{|\tilde{h}|}(\la)\sqrt{H(\la)} \le  \mathop{\rm const}
    \la^{-1+\frac{4s^2 \e}{N+2s\e}}, \notag 
\end{align}
in view of \eqref{ineq-found}, \eqref{eta-f},
\eqref{mu-estimates}, \eqref{ineq-H-upper-estimate}, \eqref{blown-up-solution},
\eqref{eq-V-1-boundary} and Proposition \ref{prop-V-H1-bounded}.
Collecting estimates 
\eqref{eq-Fourier:8}, \eqref{eq-Fourier:9} and \eqref{eq-Fourier:10} we
deduce that, for every $r\in (0,r_0]$,
\begin{equation}\label{eq-Fourier:11}
  \int_0^{r}\rho^{-m_0-N-1+2s}|\Upsilon_{m_0,k}(\rho)| \, d\rho \le
  \mathop{\rm const} \left(r+ \int_0^{r}\rho^{-1+\frac{4s^2
        \e}{N+2s\e}}
    \, d\rho\right) \le  \mathop{\rm const} r^{\frac{4s^2 \e}{N+2s\e}},
\end{equation}
thus proving \eqref{eq-Fourier:7}. Moreover we have that
\begin{equation}\label{eq-Fourier:12}
\int_0^{r_0}\rho^{m_0-1}|\Upsilon_{m_0,k}(\rho)|\, d\rho <  +\infty,
\end{equation}
as a consequence of \eqref{eq-Fourier:7}, since in a
  neighbourhood of 0, $\rho^{m_0-1}\leq\rho^{-m_0-N-1+2s}$.

Now we claim that, for every $r \in (0,r_0]$,
\begin{equation}\label{eq-Fourier:13}
c_{m_0,k}(r)+\int_0^r\rho^{m_0-1}\Upsilon_{m_0,k}(\rho) \, d\rho=0
\end{equation}
To prove \eqref{eq-Fourier:13} we argue by contradiction. If there
exists $r\in (0,r_0]$ such that \eqref{eq-Fourier:13} does not hold
true, then by \eqref{eq-Fourier:6}, \eqref{eq-Fourier:7} and
\eqref{eq-Fourier:12}
\begin{equation*}
  \varphi_{m_0,k}(\la) \sim \frac{m_0\la^{-m_0-N+2s}}{2m_0+N-2s}
  \left(c_{m_0,k}(r)+\int_0^r\rho^{m_0-1}\Upsilon_{m_0,k}(\rho) \,
    d\rho\right)
  \quad \text{as } \la \to 0^+.
\end{equation*}
From this, it follows that
\begin{equation}\label{eq-Fourier:15}
\int_0^{r_0}\la^{N-1-2s}|\varphi_{m_0,k}(\la)|^2 d \la =+\infty,
\end{equation}
since $N-2s+2m_0>0$.  On the other hand, from
\eqref{Fourier-coefficents}, the Parseval identity and
\eqref{ineq-Hardy} we deduce the following estimate
\begin{multline*}
  \int_0^{r_0}\la^{N-1-2s}|\varphi_{m_0,k}(\la) |^2 \, d \la
    \leq\int_0^{r_0}\la^{N-1-2s}\left(\int_{\mathbb{S}^+}
    \theta^{1-2s}_{N+1}
    |W(\la \theta)|^2  \, dS \right) d \la \\
  = \int_0^{r_0}\la^{-2}\left(\int_{S_\la^+} t^{1-2s}|W|^2 \, dS
    \right)d \la
    = \int_{B_{r_0}^+} t^{1-2s}\frac{|W(z)|^2}{|z|^2}\, dz< + \infty, 
  \end{multline*}
  which contradicts \eqref{eq-Fourier:15}. Hence \eqref{eq-Fourier:13}
  is proved.  From \eqref{eq-Fourier:13} and
  \eqref{eq-Fourier:11} it follows that, for every $r \in (0,r_0]$,
\begin{multline}\label{eq-Fourier:17}
  \la^{-m_0-N+2s}\left|c_{m_0,k}(r)+\int_\la^r\rho^{m_0-1}\Upsilon_{m_0,k}(\rho)
    \, d\rho\right|
  =\la^{-m_0-N+2s}\left|\int_0^\la\rho^{m_0-1}\Upsilon_{m_0,k}(\rho)
    \, d\rho\right|\\
  \le\la^{-m_0-N+2s}\left(\la^{2m_0+N-2s}\int_0^\la\rho^{-m_0-N-1+2s}|\Upsilon_{m_0,k}(\rho)|
    \, d\rho\right) \le \mathrm{const}\,
  \la^{m_0+\frac{4s^2\varepsilon}{N+2s\varepsilon}}.
\end{multline}
We finally deduce \eqref{eq-Fourier-coefficents} combining
\eqref{eq-Fourier:6}, \eqref{eq-Fourier:13} and \eqref{eq-Fourier:17}.
\end{proof}

\begin{proposition}
Let $\gamma$ be as in \eqref{limit-N}. Then
\begin{equation}\label{limit-H-positve}
\lim_{\la \to 0^+}\la^{-2\gamma}H(\la)>0.
\end{equation}
\end{proposition}
\begin{proof}
By \eqref{mu-O}, the Parseval identity and \eqref{Fourier-coefficents}
we have that 
\begin{equation}\label{eq-limit-positive:1}
H(\la)=\int_{\mathbb{S}^+} \theta_{N+1}^{1-2s}
\mu(\la\theta)|W(\la\theta)|^2 \, dS
=(1+O(\la))\sum_{m=1}^\infty \sum_{k=1}^{M_m}|\varphi_{m,k}(\la)|^2.
\end{equation}
Let $m_0\in \mathbb{N} \setminus \{0\}$ be such that $\gamma=m_0$
according to Proposition \ref{prop-blow-up-sub}. We argue by
contradiction and assume that
$0=\lim_{\la \to 0^+}\la^{-2\gamma}H(\la)=\lim_{\la \to 0^+}\la^{-2m_0}H(\la)$.
In view of \eqref{eq-limit-positive:1} this would imply that
\begin{equation*}
  \lim_{\la \to 0^+}\la^{-m_0}\varphi_{m_0,k}(\la)=0
  \quad \text{for every } k \in \{1, \dots,M_{m_0}\}.
\end{equation*}
Therefore, from \eqref{eq-Fourier-coefficents} it follows that, for
all $k \in \{1, \dots,M_{m_0}\}$ and $r\in (0,r_0]$,
\begin{multline*}
  \frac{\varphi_{m_0,k}(r)}{r^{m_0}}+\frac{m_0r^{-2m_0-N+2s}}{2m_0+N-2s}
  \int_0^r\rho^{m_0-1}\Upsilon_{m_0,k}(\rho) \, d\rho\\
  +\frac{m_0+N-2s}{2m_0+N-2s}\int_{0}^r\rho^{-m_0-N-1+2s}\Upsilon_{m_0,k}(\rho)
  \, d\rho=0,
\end{multline*}
so that, substituting into \eqref{eq-Fourier-coefficents}, we obtain
that
\begin{equation*}
  \varphi_{m_0,k}(\la)=-\frac{m_0+N-2s}{2m_0+N-2s}\la^{m_0}
  \int_{0}^\la\rho^{-m_0-N-1+2s}\Upsilon_{m_0,k}(\rho)
  \, d\rho
  +O\left(\la^{m_0+\frac{4s^2\e}{N+2s\e}}\right) 
\end{equation*}
as  $\la \to 0^+$.
Hence, from \eqref{eq-Fourier:11} we infer that
\begin{equation}\label{eq-limit-positive:6}
  \varphi_{m_0,k}(\la)=O\left(\la^{m_0+\frac{4s^2\e}{N+2s\e}}\right)
  \quad \text{ as } \la \to 0^+ \quad \text{for all } k \in \{1, \dots,M_{m_0}\}.
\end{equation}
Moreover, estimate \eqref{ineq-H-lower-estimate}
with $\sigma=\frac{2s^2\e}{N+2s\e}$ implies that 
\begin{equation}\label{eq:hso}
  \frac{1}{\sqrt{H(\la)}}=O\left(\la^{-m_0-\frac{2s^2\varepsilon}{N+2s\varepsilon}}
  \right)\quad
  \text{ as } \la \to 0^+.
\end{equation}
Since 
\begin{equation*}
\varphi_{m_0,k}(\la)=\sqrt{H(\la)}\int_{\mathbb{S}^+}\theta_{N+1}^{1-2s} V^\la(\theta) Y_{m_0,k}(\theta) \, dS \quad \text{for all } k \in \{1, \dots,M_{m_0}\}
\end{equation*}
by \eqref{Fourier-coefficents} and \eqref{blown-up-solution}, from 
\eqref{eq-limit-positive:6} and \eqref{eq:hso} we deduce that
\begin{equation}\label{eq-limit-positive:8}
  \int_{\mathbb{S}^+}\theta_{N+1}^{1-2s} V^{\la}(\theta) \Psi(\theta)\,
  dS=O\left(\la^{\frac{2s^2\e}{N+2s\e}}\right) \quad \text{ as } \la \to 0^+,
\end{equation}
for every $\Psi \in \mathop{\rm Span}\{Y_{m_0,k}: k\in\{1,\dots M_{m_0}\}\}$. By
\eqref{dimension-egigenspaces}, \eqref{orthonormal-base},
\eqref{trace-Sr+} and Proposition \ref{prop-blow-up-sub}, for any
sequence $\lambda_n\rightarrow 0^+$, there exist a
subsequence $\la_{n_h} \to 0^+$ and
$\Psi \in \mathop{\rm Span}\{Y_{m_0,k}: k\in\{1,\dots M_{m_0}\}\}$ such that
$\norm{\Psi}_{L^2(\mathbb{S}^+,\theta_{N+1}^{1-2s})}=1$ and
\begin{equation*}
\lim_{h \to +\infty}\int_{\mathbb{S}^+}\theta_{N+1}^{1-2s}
V^{\la_{n_h}}(\theta)
\Psi(\theta) \, dS=\int_{\mathbb{S}^+}\theta_{N+1}^{1-2s} |\Psi|^2 \, dS=1,
\end{equation*}
thus contradicting \eqref{eq-limit-positive:8}.
\end{proof}

\begin{theorem}\label{theor-blow-up-straight}
  Let $W$ be a non trivial weak solution to
  \eqref{prob-extension-straith}. Let $\gamma$ be as in
  \eqref{limit-N} and $m_0\in \mathbb{N}\setminus\{0\}$ be such
  that $\gamma=m_0$, according to Proposition
  \ref{prop-blow-up-sub}. Let
  $\{Y_{m_0,k}\}_{k\in\{1,\dots,M_{m_0}\}}$ be as in
  \eqref{orthonormal-base}, with $V_{m_0}$ and $M_{m_0}$ 
  defined as in \eqref{Eigenspaces} and
  \eqref{dimension-egigenspaces} respectively. Then
\begin{equation*}
\la^{-m_0}W(\la z) \to |z|^{m_0}\sum_{k=1}^{M_{m_0}} \beta_k
Y_{m_0,k}\left(\frac{z}{|z|}\right)
\quad \text{as }\la \to 0^+ \quad \text{strongly in } H^1(B_1^+,t^{1-2s}),
\end{equation*}
where $(\beta_1,\dots,\beta_{M_{m_0}})\neq (0,\dots, 0)$ and, for
every $k\in\{1,\dots,M_{m_0}\}$,
\begin{multline}\label{beta-k}
\beta_k=\frac{\varphi_{m_0,k}(r)}{r^{m_0}}+\frac{m_0r^{-2m_0-N+2s}}{(2m_0+N-2s)}\int_0^r\rho^{m_0-1}\Upsilon_{m_0,k}(\rho) \, d\rho\\
+\frac{m_0+N-2s}{2m_0+N-2s}\int_{0}^r\rho^{-m_0-N-1+2s}\Upsilon_{m_0,k}(\rho) \, d\rho,
\end{multline}
for all $r \in (0,r_0]$, where $\varphi_{m_0,k} $ is defined
in \eqref{Fourier-coefficents} and $\Upsilon_{m_0,k}$ in
\eqref{Upsilon} .
\end{theorem}
\begin{proof}
  From Proposition \ref{prop-blow-up-sub}, \eqref{orthonormal-base},
  and \eqref{limit-H-positve} it follows that, for any sequence
  $\{\la_n\}$ such that $\la_n \to
0^+$ as $n\to\infty$, there exist a subsequence
  $\{\la_{n_h}\}$ and real numbers $\beta_1, \dots, \beta_{M_{m_0}}$
  such that 
  $(\beta_1, \dots, \beta_{M_{m_0}})\neq (0,\dots, 0)$ and
\begin{equation}\label{eq:limit-blow-up-straight:1}
  \la_{n_h}^{-m_0}W(\la_{n_h} z) \to |z|^{m_0}\sum_{k=1}^{M_{m_0}}
  \beta_k Y_{m_0,k}
  \left(\frac{z}{|z|}\right) \quad \text{as }h \to +\infty \quad \text{strongly in } H^1(B_1^+,t^{1-2s}).
\end{equation}
We claim that the numbers $\beta_1, \dots \beta_{M_{m_0}}$ depend
neither on the sequence $\{\la_n\}$ nor on its subsequence
$\{\la_{n_h}\}$.  Letting $\varphi_{m_0,k}$ be as
\eqref{Fourier-coefficents}, for every $k \in \{1, \dots,M_{m_0}\}$
\begin{equation}\label{eq:limit-blow-up-straight:2}
  \lim_{h \to +\infty} \la_{n_h}^{-m_0}\varphi_{m_0,k}(\la_{n_h})=
  \lim_{h \to +\infty} \int_{\mathbb{S}^+}\theta_{N+1}^{1-2s}
  \la_{n_h}^{-m_0}W(\la_{n_h}\theta)Y_{m_0,k}(\theta)\,dS=\beta_k,
\end{equation}
thanks to \eqref{eq:limit-blow-up-straight:1} and the compactness of
the trace operator in \eqref{trace-Sr+}.  Combining
\eqref{eq:limit-blow-up-straight:2} and \eqref{eq-Fourier-coefficents}
we obtain that, for every $r \in (0,r_0]$, $\beta_k=\lim_{h \to +\infty}
  \la_{n_h}^{-m_0}\varphi_{m_0,k}(\la_{n_h})$ is equal to the right
  hand side in \eqref{beta-k},
thus proving the claim.  By Urysohn's subsequence principle we
conclude that the convergence in \eqref{eq:limit-blow-up-straight:1} holds
as $\la \to 0^+$, hence the proof is complete.
\end{proof}

\section{Proofs of the main results} \label{sec-proofs-of-the-main-results}

The proof of Theorem \ref{theor-blow-up-extended} is obtained as
  a consequence of the following result.
\begin{theorem}\label{theor-blow-up-extended-all}
  Let $N>2s$ and $\Omega\subset \R^N$ be a bounded Lipschitz
  domain such that $0\in\partial \Omega$ and \eqref{hypo-g}--\eqref{hypo-g-Omega} are
  satisfied with $x_0=0$ for some function $g$ and $R>0$. Let $U$ be
  a non trivial solution to \eqref{prob-extension} in the sense of
  \eqref{eq-weak}, with $h$ satisfying \eqref{hypo-h}, and let
    \begin{equation}\label{eq:Uhat}
      \widehat U(z)=
      \begin{cases}
        U(z),&\text{if }z\in \mathcal C_\Omega\cap F(B_{r_0}^+),\\
        0,&\text{if }z\in F(B_{r_0}^+)\setminus \mathcal C_\Omega,
      \end{cases}
    \end{equation}
    with $F$ and $r_0$ being as in Proposition \ref{diffeomorphism}.
Then there
  exist $m_0\in \mathbb{N} \setminus \{0\}$ (which is odd
  in the case $N=1$) such that
\begin{multline}\label{limit-blow-up-extended-all}
	\la^{-m_0}\widehat{U}(\la z) \to |z|^{m_0}\sum_{k=1}^{M_{m_0}}
        \beta_k \widehat Y_{m_0,k}\left(\frac{z}{|z|}\right) \quad
        \text{as }\la \to 0^+
        \quad \text{strongly in } H^1(B_1^+,t^{1-2s}),
\end{multline}
where  $M_{m_0}$ is  as
in \eqref{dimension-egigenspaces},
  \begin{equation}\label{eq:hatY}
  \widehat Y_{m_0,k} (\theta',\theta_N,\theta_{N+1})=
  \begin{cases}
      Y_{m_0,k}(\theta',\theta_N,\theta_{N+1}),&\text{if }\theta_N<0,\\
     0,&\text{if }\theta_N\geq0,
  \end{cases}
\end{equation}
with $\{Y_{m_0,k}\}_{k\in\{1,\dots,M_{m_0}\}}$ being as in
\eqref{orthonormal-base}, and the coefficients $\beta_k$ satisfy
\eqref{beta-k}.
\end{theorem}

\begin{proof}
  If $U$ is a non trivial solution of \eqref{prob-extension}, then the
  function $W$ defined in \eqref{eq:1defW} and \eqref{W} belongs to 
  $H^1(B^+_{r_0},t^{1-2s})$ and is a non trivial weak solution to
  \eqref{prob-extension-straith}.
  Letting
 \begin{equation*}
  \widehat W (z)=
  \begin{cases}
      W (z),&\text{if }z\in \mathcal Q_{r_0},\\
     0,&\text{if }z\in B_{r_0}^+\setminus\mathcal Q_{r_0},
  \end{cases}
\end{equation*}
where $\mathcal Q_{r_0}$ is defined in \eqref{eq:calQ}, by Remark \ref{null-hyperplane}
we have that $\widehat W\in H^1(B^+_{r_0},t^{1-2s})$. Moreover Theorem
\ref{theor-blow-up-straight} implies that
\begin{equation*}
  \la^{-m_0}\widehat{W}(\la z) \to
  \widehat \Phi(z) \quad \text{strongly in } H^1(B_1^+,t^{1-2s})  \quad
        \text{as }\la \to 0^+,
\end{equation*}
where
\begin{equation*}
 \widehat \Phi(z)=
  |z|^{m_0}\sum_{k=1}^{M_{m_0}}
        \beta_k \widehat Y_{m_0,k}\left(\frac{z}{|z|}\right)
      \end{equation*}
      with  $\beta_k$ as in \eqref{beta-k}.
      Hence, by homogeneity,
      \begin{equation}\label{eq:congpra}
  \la^{-m_0}\widehat{W}(\la z) \to
  \widehat \Phi(z) \quad \text{strongly in } H^1(B_r^+,t^{1-2s})  \quad
        \text{as }\la \to 0^+\quad\text{for all }r>1.
\end{equation}
      We note that
      \begin{equation}\label{eq:res}
        \la^{-m_0}\widehat{U}(\la z)=\la^{-m_0}\widehat{W}(\la G_\lambda(z))
        \quad \text{ and } \quad \nabla \left(\frac{\widehat U(\la
            \cdot)}{\la^{m_0}}\right)
        =\nabla \left(\frac{\widehat W(\la
            \cdot)}{\la^{m_0}}\right)(G_\la(z))J_{G_\la}(z)
      \end{equation}
where 
\begin{equation*}
G_\la(z):=\frac{1}{\la}F^{-1}(\la z) \quad \text{for any }  \la
\in(0,1] \text{ and }z \in \frac1\lambda F(B_{r_0^+}).
\end{equation*}  
From Proposition \ref{diffeomorphism} we deduce that 
\begin{equation*}
  G_\la(z)=z +O(\la) \quad \text{and} \quad J_{G_\la}(z)
  =\mathop{\rm Id}\nolimits_{N+1}+O(\la)\quad \text{as } \la \to 0^+ 
\end{equation*}
uniformly respect to $z \in B_1^+$.  It follows that, if $f_\la \to f$
in $L^2(B_{r}^+,t^{1-2s})$ as $\la \to 0^+$ for some
  $r>1$, then
$f_\la\circ G_\la \to f$ in $L^2(B_1^+,t^{1-2s})$ as $\la\to
0^+$. Then we  conclude in view of \eqref{eq:congpra} and \eqref{eq:res}.
\end{proof}
\begin{proof}[Proof of Theorem \ref{theor-blow-up-extended} ]
  It follows directly from Theorem \ref{theor-blow-up-extended-all} up
  to a translation.
\end{proof}
Passing to traces in \eqref{limit-blow-up-extended-all} we obtain
the following blow-up result for solutions to \eqref{eq-spec-lapla}.

\begin{theorem}\label{theor-blow-up-down-all}
  Let $N>2s$ and $\Omega\subset \R^N$ be a bounded Lipschitz domain
  such that $0\in\partial \Omega$ and
  \eqref{hypo-g}--\eqref{hypo-g-Omega} are satisfied with $x_0=0$ for
  some function $g$ and $R>0$.  Let $u\in \mathbb{H}^s(\Omega)$ be a
  non trivial solution of \eqref{eq-spec-lapla} in the sense of
  \eqref{eq-spec-lapla-weak}, with $h$ satisfying \eqref{hypo-h},
  and let $\widehat u(x)=\iota(u)$ with $\iota$ defined in
    \eqref{eq:triv-ext}. Then there exists
  $m_0 \in \mathbb{N} \setminus \{0\}$ (which is odd in the case
  $N=1$) such that
  \begin{equation*} \la^{-m_0}\widehat{u}(\la
      x) \to |x|^{m_0}\sum_{k=1}^{M_{m_0}} \beta_k \widehat
      Y_{m_0,k}\left(\frac{x}{|x|},0\right) \quad \text{as }\la \to
      0^+ \quad \text{strongly in } H^s(B'_1),
\end{equation*}
where  $M_{m_0}$ is  as
in \eqref{dimension-egigenspaces},
 $\{\widehat Y_{m_0,k}\}_{k\in\{1,\dots,M_{m_0}\}}$ are defined in
\eqref{eq:hatY} and the coefficients $\beta_k$ satisfy \eqref{beta-k}.
\end{theorem}
\begin{proof}
As observed in \cite{MR2825595} and recalled at page
    \pageref{ext}, if $u\in \mathbb{H}^s(\Omega)$ is a non trivial
    solution of \eqref{eq-spec-lapla}, then its extension
    $\mathcal H(u)=U$ is non trivial solution to \eqref{prob-extension}. Hence the
    corresponding function $\widehat U$ defined in \eqref{eq:Uhat}
    satisfies \eqref{limit-blow-up-extended-all} by Theorem
    \ref{theor-blow-up-extended-all}.  Since
    $\widehat u=\mathop{\rm Tr}(\widehat U)$, the conclusion follows
    from Proposition \ref{prop-traces}.
\end{proof}
\begin{proof}[Proof of Theorem \ref{theor-blow-up-down}]
  It follows directly from Theorem \ref{theor-blow-up-down-all} up to a translation.
\end{proof}

\appendix 
\section{Neumann eigenvalues on the half-sphere under a symmetry
    condition} \label{appendix-eigenvalues-half-sphere}

In order to determine the eigenvalues of \eqref{prob-eigenvalues}, we
first need the following preliminary lemma.

\begin{lemma}\label{lemma-polynomalian-solution}
  Let $m, N \in \mathbb{N} \setminus \{0\}$ and let $u \in C^m(\R^N)\setminus\{0\}$
  be a positively homogeneous function of degree $m$, i.e.
\begin{equation}\label{homogenity}
	u(\la x)=\la^m u(x) \quad \text{for every } \la > 0 \text{ and } x \in \R^N.
\end{equation}
Then $u$ is a homogeneous polynomial of degree $m$. 
\end{lemma}
\begin{proof}
  Let $\alpha=(\alpha_1,\dots,\alpha_N) \in \mathbb{N}^N$ be a
  multindex, $|\alpha|:=\sum_{i=1}^N \alpha_i$, and
  $x^\alpha=x_1^{\alpha_1} \dots x_N^{\alpha_N}$ for any
  $x=(x_1,\dots,x_N) \in \R^N$. By Taylor's Theorem with Lagrange
  remainder centered at $0$, for any $x \in \R^{N}$ there exists
  $t \in [0,1]$ such that
\begin{equation*}
u(x)= \sum_{|\alpha| <m} c_\alpha\pd{^{|\alpha|} u}{x^\alpha}(0)
x^\alpha
+\sum_{|\alpha| =m}c_\alpha\pd{^{|\alpha|} u}{x^\alpha}(tx) x^\alpha,
\end{equation*}
where $c_\alpha>0$ are positive constants depending on $\alpha$ and
$\pd{^{|\alpha|} u}{x^\alpha}$ stands for
$\frac{\partial^{|\alpha|} u}{\partial x_1^{\alpha_1}\cdots \partial
  x_N^{\alpha_N}}$.  By \eqref{homogenity}, one can easily prove that
$\pd{^{|\alpha|} u}{x^\alpha}$ is a positively homogeneous function of
degree $m-|\alpha|$ for all $\alpha$ with $|\alpha|\leq m$.  Thus,
combining this fact with the continuity of
$\pd{^{|\alpha|} u}{x^\alpha}$, it is clear that
$\pd{^{|\alpha|} u}{x^\alpha}(0)=0$ for every $\alpha\in \mathbb{N}^N$
with $|\alpha|<m$. On the other hand, for every
$\alpha\in \mathbb{N}^N$ with $|\alpha|=m$, we have that
$\pd{^{|\alpha|} u}{x^\alpha}$ is constant and exactly equal to
$\frac{\partial^{|\alpha|} u}{\partial x^\alpha}(0)$, being a homogeneous
function of degree $0$. It follows that
\begin{equation*}
  u(x)= \sum_{|\alpha| =m}c_\alpha\pd{^{|\alpha|} u}{x^\alpha}(0)
  x^\alpha \quad
  \text{for every } x \in \R^N,
\end{equation*}
hence proving the claim.
\end{proof}

\begin{proposition}
  All the eigenvalues of problem \eqref{prob-eigenvalues} are
  characterized by formula \eqref{eigenvalues}.
\end{proposition}
\begin{proof}
  We start by proving that if $\mu$ is an eigenvalue of
  \eqref{prob-eigenvalues}, then $\mu= m^2+m(N-2s)$ for some
  $m\in \mathbb{N}\setminus \{0\}$. If $\mu$ is an eigenvalue, then
  there exists a non trivial solution $Y$ of \eqref{prob-eigenvalues}.
  A direct computation shows that  $Y$ is a weak solution to
  \eqref{prob-eigenvalues} if and only if the function
\begin{equation*}
U(z):=|z|^\gamma Y\left(\frac{z}{|z|}\right), \quad z \in \R_+^{N+1}, 
\end{equation*}
with 
\begin{equation}\label{gamma}
\gamma:=-\frac{N-2s}{2}+\sqrt{\left(\frac{N-2s}2{}\right)^2 +\mu},
\end{equation}
belongs to $H^1_{\mathrm{loc}}(\R_+^{N+1},t^{1-2s})$, is odd with
respect to $y_N$ and weakly solves
\begin{equation}\label{prob-div-0}
\begin{cases}
\dive(t^{1-2s}\nabla U)=0, &\text{in }\R_+^{N+1},\\
\lim_{t \to 0^+}t^{1-2s} \pd{U}{\nu}=0, &\text{on } \R^N.
\end{cases}	
\end{equation}
Hence, if $\mu$ is an eigenvalue of \eqref{prob-eigenvalues}, there
exists a solution $U$ of \eqref{prob-div-0} which is odd with respect
to $y_N$ and positively homogeneous of degree $\gamma$. The regularity
result in \cite[Theorem 1.1]{STV} ensures that
$U\in C^{\infty}(\overline{B_1^+})$. Then there exists
$m \in \mathbb{N}\setminus \{0\}$ such that $\gamma =m$ and so
$\mu=m^2+m(N-2s)$ thanks to \eqref{gamma}. We notice that the case $m=0$
is excluded since in that case $\mu=0$ and 0 is not an eigenvalue.
Indeed, if by contradiction 0
is an eigenvalue, letting $Y$ be an eigenfunction of
\eqref{prob-eigenvalues} with associated eigenvalue $0$ and choosing in
\eqref{eq-egienvlulues} $\Psi=Y$, we would have that $Y$ is constant and
$Y \not\equiv 0$, hence
$Y \notin H_{\rm odd}^1(\mathbb{S}^+,\theta_{N+1}^{1-2s})$ which is a
contradiction (see \eqref{H1d}).

Viceversa, in order to prove that the numbers given in
\eqref{eigenvalues} are eigenvalues of \eqref{prob-eigenvalues}, we
need to show that, for any fixed
  $m \in \mathbb{N} \setminus \{0\}$, there actually exist an
eigenfunction associated to $m^2+m(N-2s)$ if $N>1$ and an
eigenfunction associated to $(2m-1)^2+(2m-1)(N-2s)$ if
$N=1$. Equivalently, for any fixed $m \in \mathbb{N} \setminus \{0\}$
we have to find a non trivial solution to \eqref{prob-div-0} which is
odd with respect to $y_N$ and positively homogeneous with degree $m$
if $N>1$ and $2m-1$ if $N=1$.
To this end, we observe that equation $\dive(t^{1-2s}\nabla U)=0$ can
be rewritten as
\begin{equation}\label{eq:eginvluaes-computation:1}
\Delta U +\frac{1-2s}{t}U_t=0. 
\end{equation}
We
first consider the case $N=1$.  If $n =2m -1$ with
$m\in \mathbb{N}\setminus\{0\}$, we consider the following
homogeneous polynomial of degree $2m-1$, odd with respect to
$y_1$,
\begin{equation}\label{solution-m-odd-N=1}
  U_{1,m}(y_1,t):=\sum_{k=0}^{m-1} a_k y_1^{2k +1 } t^{2m-2k-2},
\end{equation} 
with $a_0, \dots, a_{m-1} \in \mathbb{R}$. A direct computation shows that
$U_{1,m}$ is a solution of \eqref{prob-div-0}, and equivalently of 
\eqref{eq:eginvluaes-computation:1}, if and only if
\begin{equation*}
  a_{k}=\frac{-2[(m-k)^2-s(m-k)]}{k(2k+1)} a_{k-1} \quad
    \text{for all } k \in \{1,\dots,m-1\}.
\end{equation*}
Thus, for example choosing $a_0:=1$, we have constructed a non trivial
solution to \eqref{prob-div-0} which is odd with respect to $y_1$ and
 positively homogeneous of degree $2m-1$.

To complete the proof of \eqref{eigenvalues} in the case $N=1$, it
remains to show that, if $n=2m$ with  $m\in \mathbb{N}\setminus
\{0\}$, then
$n^2+n(N-2s)$ is not an eigenvalue of \eqref{prob-eigenvalues}. To
this aim, we argue by contradiction and assume that $(2m)^2+2m(N-2s)$ is
an eigenvalue of \eqref{prob-eigenvalues} associated to an
eigenfunction $\Psi$. Then the function defined as 
\begin{equation*}
U(z)=|z|^\gamma\Psi\left(\frac{z}{|z|}\right), \quad z=(y_1,t)\in \mathbb{R}^2_+,
\end{equation*}
with 
\begin{equation*}
\gamma=-\frac{N-2s}{2}+\sqrt{\left(\frac{N-2s}{2}\right)^2+(2m)^2+2m(N-2s)}=2m
\end{equation*}
is a non trivial solution to \eqref{prob-div-0}, odd with
respect to $y_1$. Hence, if we consider the even reflection of $U$
with respect to $t$, namely the function
$\widetilde{U}(y_1,t):=U(y_1,|t|)$, we have that $\widetilde{U}$ is a
solution of $\mathrm{div}(|t|^{1-2s}\nabla \widetilde{U})=0$ in
$\mathbb{R}^{2}$. Then, by \cite[Theorem 1.1]{STV} we deduce that
$\widetilde{U}\in C^\infty(\mathbb{R}^{2})$. Moreover, $\widetilde{U}$
is positively homogeneous of degree $\gamma=2m$, therefore from Lemma
\ref{lemma-polynomalian-solution} it follows that $\widetilde{U} $ is
a homogeneous polynomial of degree $2m$, namely
\begin{equation*}
\widetilde{U}(y_1,t)= \sum_{k=0}^{2m}a_k y_1^{2m-k} t^k
\end{equation*}
where $a_k=0$ if $k$ is odd since $\widetilde{U}$ is even with
respect to $t$.  In this way $\tilde{U}$ turns out to be even also
with respect to $y_1$ and this contradicts the fact that $U$ is
non trivial and odd with respect to $y_1$.

If $N=2$ and $m \in \mathbb{N}\setminus\{0\}$ is odd, then we
consider $U_2(y_1,y_2,t):=U_{1,n}(y_2,t)$, where $U_{1,n}$ is defined  in
\eqref{solution-m-odd-N=1} and 
$n\in \mathbb{N}\setminus\{0\}$ is such that $m=2n-1$. Such $U_2$ is a positively
homogeneous solution of \eqref{prob-div-0} of degree $m$, odd with
respect to $y_2$. If $m \in \mathbb{N}\setminus\{0\}$ is even,
  i.e. $m =2n $ with $ n \in \mathbb{N} \setminus \{0\}$,
then we define
\begin{equation*}
U_3(y_1,y_2,t):=\sum_{k=0}^{n-1} a_k y_1^{2k +1 } y_2^{2n-2k-1},
\end{equation*}
with $a_0, \dots, a_{n-1} \in \mathbb{R}$.
A direct computation shows that
$U_3$ is a solution of \eqref{prob-div-0}, and equivalently of 
\eqref{eq:eginvluaes-computation:1}, if and only if
\begin{equation*}
  a_{k+1}=\frac{-[2(n-k)^2-3n+3k+1]}{(2k^2+5k+3)} a_k \quad \text{for
    all } k \in \{0,\dots,n-2\}.
\end{equation*}
Then, choosing for example again $a_0=1$, we obtain that $U_3$ is a solution of
\eqref{prob-div-0} which is positively homogeneous of degree $m$ and
odd with respect to $y_2$, as desired.

If $N>2$, for any $m \in \mathbb{N}\setminus \{0\}$ there exists a
harmonic homogeneous polynomial $P\not\equiv0$ in the variables
$y_1, \dots, y_{N-1},$ of degree $m-1$. Then
$U_4(y_1,\dots,y_{N-1},y_N,t):= P(y_1,\dots,y_{N-1})\,y_N$ is a non
trivial solution to \eqref{prob-div-0} which is odd with respect to $y_N$
and positively homogeneous of degree~$m$.

\end{proof}

\end{document}